\documentclass[12pt]{amsart}

\usepackage{amsmath,amssymb}
\usepackage[latin1]{inputenc}
\usepackage[dvips]{graphics}
\input{xy}
\xyoption{poly}
\xyoption{2cell}
\xyoption{all}
\usepackage{epsfig}  
\usepackage{pstricks}
\usepackage{color}	

\usepackage{tocvsec2}

\newcommand{\overunder}[2]{
\!\begin{array}{c}
\scriptstyle{#1}\\[-.1in]
-\!\!\!-\!\!\!-\\[-.1in]
\scriptstyle{#2}
\end{array}
\!
}

\def\sgn{\operatorname{sgn}}
\def\TT{\mathbb{T}}

\def\Trop{\operatorname{Trop}}
\def\QQ{\mathbb{Q}}
\def\Q{\mathbb{Q}}
\def\B{\mathcal{B}}
\def\BB{\mathbb{B}}
\def\gg{\mathbf{g}}

\newcommand{\za}{\alpha}
\newcommand{\zb}{\beta}
\newcommand{\zd}{\delta}
\newcommand{\zD}{\Delta}
\newcommand{\ze}{\epsilon}
\newcommand{\zg}{\gamma}

\newcommand{\zl}{\lambda}
\newcommand{\zs}{\sigma}

\DeclareMathOperator{\Bang}{Bang}
\DeclareMathOperator{\Good}{Good}
\DeclareMathOperator{\Brac}{Brac}

\DeclareMathOperator{\cross}{cross}

\newcommand{\taubar}{\overline{\tau}}

\newcommand{\C}{\mathcal{C}}
\newcommand{\A}{\mathcal{A}}

\newcommand{\PP}{\mathbb{P}}
\newcommand{\x}{\mathbf{x}}

\newcommand{\Z}{\mathbb{Z}}

\newcommand{\Y}{\mathsf{y}}
\newcommand{\X}{\mathsf{x}}

\newcommand{\tbar}{\overline{\tau}}
\newcommand{\sbar}{\overline{\sigma}}

\def\Aprin{\Acal_\bullet}
\def\Afull{\Acal_{*}}
\def\Xcal{\mathcal{X}}
\def\Acal{\mathcal{A}}
\def\Fcal{\mathcal{F}}

\def\yy{\mathbf{y}}

\def\xx{\mathbf{x}}

\def\Q{\mathbb{Q}}
\def\ZZ{\mathbb{Z}}

\newtheorem{theorem}{Theorem}[section]

\newtheorem{lemma}[theorem]{Lemma}

\newtheorem{prop}[theorem]{Proposition}
\newtheorem{proposition}[theorem]{Proposition}
\newtheorem{conjecture}[theorem]{Conjecture}
\newtheorem{cor}[theorem]{Corollary}
\newtheorem{corollary}[theorem]{Corollary}

\theoremstyle{definition}
\newtheorem{definition}[theorem]{Definition}
\newtheorem{Def}[theorem]{Definition}
\newtheorem{example}[theorem]{Example}

\theoremstyle{remark}
\newtheorem{remark}[theorem]{Remark}

\numberwithin{equation}{section}

\begin{document}
\title{Bases for cluster algebras from surfaces}
\author{Gregg Musiker}
\address{School of Mathematics, University of Minnesota, Minneapolis, MN 55455}
\email{musiker@math.umn.edu}
\author{Ralf Schiffler} 
\address{Department of Mathematics, University of Connecticut, 
Storrs, CT 06269-3009}
\email{schiffler@math.uconn.edu}
\thanks{{The first author is  partially
   supported by the NSF grant DMS-1067183. The second author is partially
   supported by the NSF grant DMS-1001637.   The third author
is partially supported by the NSF grant DMS-0854432, and 
an NSF CAREER award.}}
\author{Lauren Williams}
\address{Department of Mathematics, University of California,
Berkeley, CA 94720}
\email{williams@math.berkeley.edu}

%\subjclass[2000]{16S99, 05C70, 05E15}
\date{}
\dedicatory{}


\begin{abstract}  We construct
two bases for each cluster algebra coming from a triangulated
surface without punctures.
We work in the context of a coefficient system
coming from a full-rank exchange matrix,
for example, {\it principal coefficients}.
\end{abstract}

\maketitle
\maxtocdepth{subsection}
\tableofcontents

\section{Introduction}\label{intro}
Fomin and Zelevinsky introduced cluster
algebras in 2001 \cite{FZ1}, in an attempt to create an algebraic framework
for Lusztig's dual canonical bases and total positivity in semisimple groups
\cite{Lusztig1, Lusztig2, Lusztig3}.  In particular, 
writing down explicitly the elements of the dual canonical basis 
is a very difficult problem; but Fomin and Zelevinsky conjectured
that a large subset of them can be understood via the machinery
of cluster algebras.  More precisely, 
they conjectured that all monomials in the variables
of any given cluster (the \emph{cluster monomials}) belong 
to (the classical limit at $q \to 1$ of) the dual canonical basis \cite{FZ1}.
For recent progress in this direction, see \cite{GLS3,HL2,La1,La2}.

Because of the conjectural connection between cluster algebras and dual canonical bases, it is natural to ask whether one may construct a
``good'' (vector-space) basis $\B$ of each cluster algebra $\A$. In keeping with 
Fomin and Zelevinsky's conjecture, such a basis should include the 
cluster monomials.  Additionally, since the dual canonical basis 
has  striking positivity properties, a good basis of a cluster algebra should
also have analogous positivity properties.  
In particular, if we define $\A^+$ to be the set of elements of $\A$
which expand positively with respect to every cluster, then
one should require that every element $b\in \B$ is also in $\A^+$.
In the case that $b$ is a cluster variable,
this requirement is equivalent to the 
well-known {\it Positivity Conjecture}, one of the main open questions about cluster
algebras.  

The construction of bases for cluster algebras is a problem that has gained a lot of 
attention recently.
Caldero and Keller showed that for cluster algebras of finite type, the cluster monomials 
form a basis \cite{CK}.  For cluster algebras which are not of finite type, the cluster monomials
do not span the cluster algebra, but it follows from \cite{DWZ}, see also \cite{plamondon}, 
that they  are linearly independent as long as the initial exchange matrix  
of the cluster algebra has full rank.  
Sherman and Zelevinsky constructed bases containing the cluster
monomials for the cluster algebra of rank 2 affine types \cite{SZ,Z}, and 
Cerulli-Irelli did so for rank 3 affine types \cite{cerulli}. 
Dupont has used cluster categories to construct the so-called \emph{generic basis} for the affine types  
\cite{Dgeneric, Dup}, see also \cite{DXX}. Geiss, Leclerc and Schr\"oer constructed the generic basis in a much 
more general setting \cite{GLS,GLS2}, which in particular includes all acyclic cluster algebras.
Plamondon \cite[Chapter 5] {Pl2} gives a convenient reparametrization of Geiss-Leclerc-Schr\"oer's basis.

There is an important 
class of cluster algebras associated to \emph{surfaces with marked points}
\cite{FG1, FG2, GSV, FST, FT}.  Such cluster algebras are 
of interest for several reasons.  First, they have a topological 
interpretation: they may be viewed as coordinate rings of the 
corresponding \emph{decorated Teichm\"uller space} \cite{Pen1, Pen2}. Second, 
such cluster algebras  comprise
most of the \emph{mutation-finite} cluster algebras \cite{FeSTu}, that is,
the cluster algebras which have finitely many different exchange matrices.
The (generalized) cluster category of a cluster algebra from a surface has been
defined whenever the surface has a non-empty boundary \cite{BMRRT, A, ABCP, LF, CLF}.
It has been described in geometric terms in \cite{CCS} for the disk, in \cite{S1} for the disk with one puncture, 
and in \cite{BZ} for arbitrary surfaces without punctures. 

Note that the aforementioned constructions do not yield bases 
in the case of cluster algebras from surfaces, in general.

The present paper was inspired by work of Fock and Goncharov \cite{FG1},
and Fomin, Shapiro and Thurston \cite{FST2}.  
In \cite{FG1}, Fock and Goncharov introduced a canonical basis for the cluster 
varieties related to SL$_2$. 
In particular, their
construction gives a basis for the algebra of universally Laurent
polynomials in the dual space, which coincides with the (coefficient-free) \emph{upper} cluster
algebra associated to the surface.  
(Note that in general, the upper cluster algebra contains but is not equal to the cluster algebra.)
Moreover, the elements of their bases have positive Laurent expansions
in all of the clusters that they consider \cite{FG1}.
In a lecture series in 2008 \cite{FST2}, 
D. Thurston announced a construction of two bases 
associated to a cluster algebra from a surface, based on joint work with Fomin and Shapiro,
and inspired by \cite{FG1}; note however that this work was not completed. 

Both of these constructions are parameterized by the same collections 
$\C^\circ$ and $\C$ of 
curves in a surface.  
Recall that an \emph{arc} in a surface with marked points
is (the isotopy class of) a curve connecting two marked points
which has no self-crossings.
A \emph{closed loop}
is a  non-contractible closed curve which is disjoint from the boundary. A closed loop without self-crossings is called \emph{essential}. A multiset of $k$ copies of the same essential loop
is called a \emph{$k$-bangle} and a closed loop obtained by following an essential loop $k$ times, thus creating $k-1$ self-crossings, is called a \emph{$k$-bracelet}. 
Let $\C^\circ$  be the collection of multisets of arcs and essential loops which have no crossings;
and let $\C$ be obtained from $\C^\circ$ by 
replacing the maximal $k$-bangles by the corresponding $k$-bracelets.
In \cite{FG1}, the authors associated a Laurent polynomial to each collection of curves by using 
(the upper right entry or trace of) an appropriate product of elements of $SL_2$.  In \cite{FST2},
the authors associated a cluster algebra element to a collection of curves by using
the (normalized) lambda length of that collection.  These two notions coincide.

In our previous work \cite{MSW}, we gave 
combinatorial formulas for the cluster variables in the cluster algebra 
associated to any surface with marked points,
building on earlier work in \cite{S2,ST,S3,MS}. The formula for the cluster variable associated to 
an arc is a weighted sum over perfect matchings of a planar \emph{snake graph} associated to the arc.
(There are similar formulas for other cluster variables).
Since these formulas are manifestly
positive, the positivity conjecture follows as a corollary.
 
In the present paper, we generalize our formulas from \cite{MSW} to associate a 
Laurent polynomial to each collection of curves in $\C^\circ$ and $\C$ in 
an unpunctured surface $(S,M)$ (i.e. all marked points lie on the boundary). Instead
of using perfect matchings of a planar graph, the 
Laurent polynomial associated to a closed curve is a weighted sum
over \emph{good} matchings in a \emph{band graph} on a Mobius strip or annulus.
We work in the context of a cluster algebra $\A$ associated to $(S,M)$ whose
coefficient system comes from 
a full-rank exchange matrix -- for example, principal coefficients.
In this way we construct bases $\B^{\circ}$ and $\B$ for $\A$ which are 
parameterized by the collections
$\C^\circ$ and $\C$.  
Our bases are manifestly positive, in the sense that both 
$\B^{\circ}$ and $\B$ are contained in $\A^+$.
For surfaces with punctures, we still have a construction of sets $\B^\circ$ and $\B$, 
but not all of the proofs can be adapted to that case.

It is not obvious, but it is possible to show via the results of 
\cite{MW} that the bases we consider in this paper 
coincide with those considered in [FST2], and (in the coefficient-free
case) those
in [FG1].

Our main result is the following.

\begin{theorem}\label{main theorem} Let $\A$ be a cluster algebra with principal coefficients from an
unpunctured  surface, which has at least two marked points. Then
$\B^{\circ}$ and 
$\B$ are both  bases of $\A$.  Moreover,
each element of $\B^{\circ}$ and $\B$ has a positive
Laurent expansion with respect to any cluster of $\A$.
\end{theorem}

\begin{corollary}\label{cor:coefficients}
Let $\Afull$ be a cluster algebra from an unpunctured surface with at least two marked points,
whose  coefficient system  comes from a full-rank exchange matrix.  Then 
there are bases 
$\BB^{\circ}$ and  $\BB$ for $\Afull$, whose elements have positive Laurent expansions
with respect to every cluster of $\Afull$.
\end{corollary}

We are grateful to Goncharov \cite{G} for pointing out that
using the results in \cite{FG1} together with Theorem \ref{main theorem}, one may  deduce
Corollary \ref{cor:upper} (a).

\begin{corollary} \label{cor:upper}
Let $\A$ be a coefficient-free cluster algebra from an unpunctured surface
with at least two marked points.  
\begin{itemize}
\item[\textup{(a)}]  The upper cluster algebra coincides with the cluster algebra.
\item[\textup{(b)}] $\B^{\circ}$ and 
$\B$ are both  bases of $\A$.
\end{itemize}
\end{corollary}

Besides the property that $\B^{\circ}$ and $\B$ lie in $\A^+$, one might ask 
whether the structure constants for these bases are positive. In other words,
is it the case that every product of basis elements, when expanded as a linear combination
of basis elements, has all coefficients positive?

\begin{conjecture} \cite[Section 12]{FG1}, \cite{FST2} Both bases $\B^{\circ}$ and $\B$ have positive structure constants.
\end{conjecture}

As a partial result in this direction, Cerulli-Irelli and Labardini \cite{CLF}  showed that for a surface with non-empty
boundary, the elements of $\mathcal{A}^+$ that lie in the span of the set of cluster monomials have
positive structure constants.

Finally, one might ask whether either of these bases is \emph{atomic}.  
We say that $\B$ is an atomic basis for $\A$ if 
$a\in \A^+$ if and only if when we write $a = \sum_{b\in \B} \lambda_b\ b,$
every coefficient $\lambda_b$ is non-negative.  Sherman and Zelevinsky showed that the bases they constructed were atomic. They also showed that  if an atomic
basis exists, it is necessarily unique \cite{SZ}.

In the case of finite type cluster algebras, 
Cerulli-Irelli \cite{cerulli2}  showed that 
the basis of cluster monomials is in fact atomic.
Recently, Dupont and Thomas proved in \cite{DuTh} that the basis constructed by Dupont in \cite{D} for the affine $\tilde{\mathbb{A}}$ types is an atomic basis. This basis coincides with our basis $\B$ in the case where the surface is an annulus, and all coefficients are set to 1.  Their proof
uses the surface model, and 
we expect that it can be generalized to arbitrary unpunctured surfaces.  

\begin{conjecture} The basis $\B$ is an atomic basis.
\end{conjecture}

To prove Theorem \ref{main theorem} we need to show that 
$\B^{\circ}$ and $\B$ are contained in $\A$, that they form a spanning set, and that they
are linearly independent. The positivity property follows by construction
(elements are defined as sums over perfect matchings of certain graphs), together 
with \cite[Theorem 3.7]{FZ4}.
We show that  
both  $\B^\circ$ and $\B$ are spanning sets using skein relations 
with principal coefficients \cite{MW}.  In order to show
linear independence, we need to extend the notion of $\gg$-vector, defined in \cite{FZ4}, to 
$\B^{\circ}$ and $\B$.  Along the way we prove that the set of monomials in the Laurent expansions
of elements of $\B^{\circ}$ and $\B$ have the structure of a distributive lattice.
The following result, which may be interesting in its own right, then implies linear independence
of both $\B^{\circ}$ and $\B$.

\begin{theorem}\label{th:bijection}  Let $\A$ be a cluster algebra with principal coefficients from an
unpunctured  surface, which has at least two marked points. Then
the $\gg$-vector   induces bijections $\B^\circ \to \mathbb{Z}^n$ and $\B \to \mathbb{Z}^n$.
\end{theorem}

The paper is organized as follows. After recalling some background on cluster algebras in 
Section \ref{sect2}, we define the bases $\B^\circ$ and $\B$ in 
Section \ref{sec two bases}. Sections \ref{sec:main}-\ref{sec:g}
are devoted to the proof of our main result, in the context of principal coefficients.  
Corollary \ref{cor:upper} is proven at the end of Section \ref{sect span}.
In Section \ref{full-rank}, we explain how to construct
bases for cluster algebras from surfaces in which 
the coefficient system comes from a 
full-rank exchange matrix.
Finally, in Section \ref{appendix}, we briefly sketch how to extend 
our result to surfaces 
with punctures, and explain which part of the proof does not generalize easily.

\textsc{Acknowledgements:} We thank  Gr\'egoire Dupont, Sergey Fomin, Sasha Goncharov, Bernard Leclerc, Hugh Thomas,
Dylan Thurston and Andrei Zelevinsky for interesting discussions.  We are particularly grateful to 
Dylan Thurston for his inspiring lectures in Morelia, Mexico.

%%%%%%%%%%%%%%%%%%%%%%%%%%%%%%%%
%% 
%%  SECTION
%%
%%%%%%%%%%%%%%%%%%%%%%%%%%%%%%%%
\section{Preliminaries and notation}\label{sect2}
In this section, we review some notions from the theory of cluster algebras.
\subsection{Cluster algebras}\label{sect cluster algebras}
We begin by reviewing the definition of cluster algebra,
first introduced by Fomin and Zelevinsky in \cite{FZ1}.
Our definition follows the exposition in \cite{FZ4}.
Another good reference for cluster algebras is \cite{GSV-book}.

To define  a cluster algebra~$\Acal$ we must first fix its
ground ring.
Let $(\PP,\oplus, \cdot)$ be a \emph{semifield}, i.e.,
an abelian multiplicative group endowed with a binary operation of
\emph{(auxiliary) addition}~$\oplus$ which is commutative, associative, and
distributive with respect to the multiplication in~$\PP$.
The group ring~$\ZZ\PP$ will be
used as a \emph{ground ring} for~$\Acal$.
One important choice for $\PP$ is the tropical semifield; in this case we say that the
corresponding cluster algebra is of {\it geometric type}.
Let $\Trop (u_1, \dots, u_{m})$ be an abelian group (written
multiplicatively) freely generated by the $u_j$.
We define  $\oplus$ in $\Trop (u_1,\dots, u_{m})$ by
\begin{equation}
\label{eq:tropical-addition}
\prod_j u_j^{a_j} \oplus \prod_j u_j^{b_j} =
\prod_j u_j^{\min (a_j, b_j)} \,,
\end{equation}
and call $(\Trop (u_1,\dots,u_{m}),\oplus,\cdot)$ a \emph{tropical
 semifield}.
Note that the group ring of $\Trop (u_1,\dots,u_{m})$ is the ring of Laurent
polynomials in the variables~$u_j\,$.

As an \emph{ambient field} for
$\Acal$, we take a field $\Fcal$
isomorphic to the field of rational functions in $n$ independent
variables (here $n$ is the \emph{rank} of~$\Acal$),
with coefficients in~$\QQ \PP$.
Note that the definition of $\Fcal$ does not involve
the auxiliary addition
in~$\PP$.

\begin{Def}
\label{def:seed}
A \emph{labeled seed} in~$\Fcal$ is
a triple $(\xx, \yy, B)$, where
\begin{itemize}
\item
$\xx = (x_1, \dots, x_n)$ is an $n$-tuple 
from $\Fcal$
forming a \emph{free generating set} over $\QQ \PP$,
\item
$\yy = (y_1, \dots, y_n)$ is an $n$-tuple
from $\PP$, and
\item
$B = (b_{ij})$ is an $n\!\times\! n$ integer matrix
which is \emph{skew-symmetrizable}.
\end{itemize}
That is, $x_1, \dots, x_n$
are algebraically independent over~$\QQ \PP$, and
$\Fcal = \QQ \PP(x_1, \dots, x_n)$.
We refer to~$\xx$ as the (labeled)
\emph{cluster} of a labeled seed $(\xx, \yy, B)$,
to the tuple~$\yy$ as the \emph{coefficient tuple}, and to the
matrix~$B$ as the \emph{exchange matrix}.
\end{Def}

We obtain ({\it unlabeled}) {\it seeds} from labeled seeds
by identifying labeled seeds that differ from
each other by simultaneous permutations of
the components in $\xx$ and~$\yy$, and of the rows and columns of~$B$.

We  use the notation
$[x]_+ = \max(x,0)$,
$[1,n]=\{1, \dots, n\}$, and
\begin{align*}
\sgn(x) &=
\begin{cases}
-1 & \text{if $x<0$;}\\
0  & \text{if $x=0$;}\\
 1 & \text{if $x>0$.}
\end{cases}
\end{align*}

\begin{Def}
\label{def:seed-mutation}
Let $(\xx, \yy, B)$ be a labeled seed in $\Fcal$,
and let $k \in [1,n]$.
The \emph{seed mutation} $\mu_k$ in direction~$k$ transforms
$(\xx, \yy, B)$ into the labeled seed
$\mu_k(\xx, \yy, B)=(\xx', \yy', B')$ defined as follows:
\begin{itemize}
\item
The entries of $B'=(b'_{ij})$ are given by
\begin{equation}
\label{eq:matrix-mutation}
b'_{ij} =
\begin{cases}
-b_{ij} & \text{if $i=k$ or $j=k$;} \\[.05in]
b_{ij} + \sgn(b_{ik}) \ [b_{ik}b_{kj}]_+
 & \text{otherwise.}
\end{cases}
\end{equation}
\item
The coefficient tuple $\yy'=(y_1',\dots,y_n')$ is given by
\begin{equation}
\label{eq:y-mutation}
y'_j =
\begin{cases}
y_k^{-1} & \text{if $j = k$};\\[.05in]
y_j y_k^{[b_{kj}]_+}
(y_k \oplus 1)^{- b_{kj}} & \text{if $j \neq k$}.
\end{cases}
\end{equation}
\item
The cluster $\xx'=(x_1',\dots,x_n')$ is given by
$x_j'=x_j$ for $j\neq k$,
whereas $x'_k \in \Fcal$ is determined
by the \emph{exchange relation}
\begin{equation}
\label{exchange relation}
x'_k = \frac
{y_k \ \prod x_i^{[b_{ik}]_+}
+ \ \prod x_i^{[-b_{ik}]_+}}{(y_k \oplus 1) x_k} \, .
\end{equation}
\end{itemize}
\end{Def}

We say that two exchange matrices $B$ and $B'$ are {\it mutation-equivalent}
if one can get from $B$ to $B'$ by a sequence of mutations.
\begin{Def}
\label{def:patterns}
Consider the \emph{$n$-regular tree}~$\TT_n$
whose edges are labeled by the numbers $1, \dots, n$,
so that the $n$ edges emanating from each vertex receive
different labels.
A \emph{cluster pattern}  is an assignment
of a labeled seed $\Sigma_t=(\xx_t, \yy_t, B_t)$
to every vertex $t \in \TT_n$, such that the seeds assigned to the
endpoints of any edge $t \overunder{k}{} t'$ are obtained from each
other by the seed mutation in direction~$k$.
The components of $\Sigma_t$ are written as:
\begin{equation}
\label{eq:seed-labeling}
\xx_t = (x_{1;t}\,,\dots,x_{n;t})\,,\quad
\yy_t = (y_{1;t}\,,\dots,y_{n;t})\,,\quad
B_t = (b^t_{ij})\,.
\end{equation}
\end{Def}

Clearly, a cluster pattern  is uniquely determined
by an arbitrary  seed.

\begin{Def}
\label{def:cluster-algebra}
Given a cluster pattern, we denote
\begin{equation}
\label{eq:cluster-variables}
\Xcal
= \bigcup_{t \in \TT_n} \xx_t
= \{ x_{i,t}\,:\, t \in \TT_n\,,\ 1\leq i\leq n \} \ ,
\end{equation}
the union of clusters of all the seeds in the pattern.
The elements $x_{i,t}\in \Xcal$ are called \emph{cluster variables}.
The 
\emph{cluster algebra} $\Acal$ associated with a
given pattern is the $\ZZ \PP$-subalgebra of the ambient field $\Fcal$
generated by all cluster variables: $\Acal = \ZZ \PP[\Xcal]$.
We denote $\Acal = \Acal(\xx, \yy, B)$, where
$(\xx,\yy,B)$
is any seed in the underlying cluster pattern.
\end{Def}

The remarkable {\it Laurent phenomenon} asserts the following.

\begin{theorem} \cite[Theorem 3.1]{FZ1}
\label{Laurent}
The cluster algebra $\Acal$ associated with a seed
$(\xx,\yy,B)$ is contained in the Laurent polynomial ring
$\ZZ \PP [\xx^{\pm 1}]$, i.e.\ every element of $\Acal$ is a
Laurent polynomial over $\ZZ \PP$ in the cluster variables
from $\xx=(x_1,\dots,x_n)$.
\end{theorem}

\begin{remark}\label{rectangular}
In cluster algebras whose ground ring is 
$\Trop(u_1,\dots, u_{m})$ (the tropical semifield), it is convenient to replace the
matrix $B$ by an $(n+m)\times n$ matrix $\tilde B=(b_{ij})$ whose upper part
is the $n\times n$ matrix $B$ and whose lower part is an $m\times
n$ matrix that encodes the coefficient tuple via
\begin{equation}\label{eq 20}
y_k = \prod_{i=1}^{m} u_i^{b_{(n+i)k}}.
\end{equation}  
Then the mutation of the coefficient tuple in equation (\ref{eq:y-mutation}) 
is determined by the mutation
of the matrix $\tilde B$ in equation (\ref{eq:matrix-mutation}) and the formula (\ref{eq 20}); and the
exchange relation (\ref{exchange relation}) becomes
\begin{equation}\label{geometric exchange}
 x_k'=x_k^{-1} \left( \prod_{i=1}^n x_i^{[b_{ik}]_+}
\prod_{i=1}^{m} u_i^{[b_{(n+i)k}]_+} 
+\prod_{i=1}^n x_i^{[-b_{ik}]_+}
\prod_{i=1}^{m} u_i^{[-b_{(n+i)k}]_+}
\right).
\end{equation}  
\end{remark}
%%%%%%%%%%%%%%%%%%%%%%%%%%%%
%%
%% \subsection
%%
%%%%%%%%%%%%%%%%%%%%%%%%%%%%%%%%
\subsection{Cluster algebras with principal    coefficients}\label{sect principal coefficients}

Fomin and Zelevinsky introduced in \cite{FZ4} a special type of
coefficients, called \emph{principal coefficients}.

\begin{Def}
\label{def:principal-coeffs}
We say that a cluster pattern $t \mapsto (\xx_t, \yy_t,B_t)$ on $\TT_n$
(or the corresponding cluster algebra~$\Acal$)  has
\emph{principal coefficients at a vertex~$t_0$} if
$\PP= \Trop(y_1, \dots, y_n)$ and
$\yy_{t_0}= (y_1, \dots, y_n)$. \linebreak[3]
In this case, we denote $\Acal=\Aprin(B_{t_0})$.
\end{Def}

\begin{remark}
\label{rem:principal-tildeB}
Definition~\ref{def:principal-coeffs} can be rephrased
as follows: a cluster algebra~$\Acal$ has principal coefficients at a
vertex~$t_0$ if $\Acal$ is of geometric type,
and is associated with the matrix $\tilde B_{t_0}$ of order $2n \times n$
whose upper part is $B_{t_0}$,
and whose complementary (i.e., bottom) part is
the $n \times n$ identity matrix (cf.\ \cite[Corollary~5.9]{FZ1}).
\end{remark}
 
\begin{Def}
\label{def:Aprin}
Let~$\Acal$ be the cluster \linebreak[3]
algebra with principal coefficients at 
$t_0$, defined by the initial seed
$\Sigma_{t_0}=(\xx_{t_0}\,,\yy_{t_0}\,,B_{t_0})$ with
\begin{equation}
\label{eq:initial-seed}
\xx_{t_0} = (x_1, \dots, x_n), \quad
\yy_{t_0} = (y_1, \dots, y_n), \quad
B_{t_0} = B^0 = (b^0_{ij})\,.
\end{equation}
By the Laurent phenomenon, we
can express every cluster variable $x_{\ell;t}$ as a (unique)
Laurent polynomial in $x_1, \dots, x_n, y_1, \dots, y_n$; 
we denote this by 
\begin{equation}
\label{eq:X-sf}
X_{\ell;t} 
= X_{\ell;t}^{B^0;t_0}.
\end{equation}
Let $F_{\ell;t} =
F_{\ell;t}^{B^0;t_0}$
denote the Laurent polynomial obtained from $X_{\ell;t}$ by
\begin{equation}
\label{eq:F-def}
F_{\ell;t}(y_1, \dots, y_n) = X_{\ell;t}(1, \dots, 1; y_1, \dots, y_n).
\end{equation}
$F_{\ell;t}(y_1,\dots,y_n)$ turns out to be
a polynomial \cite{FZ4} and is called an \emph{F-polynomial}.
\end{Def}

\begin{proposition}\cite[Corollary 6.2]{FZ4} \label{g}
Consider any rank $n$ cluster algebra, defined by an 
$n \times n$ exchange matrix $B$,
and consider the \emph{$\gg$-vector grading} given by
$\textup{deg}(x_i)=\mathbf{e}_i$ and
$\textup{deg}(y_j)=-\mathbf{b}_j$, where
$\mathbf{e}_i=(0,\ldots,0,1,0,\ldots,0) \in\mathbb{Z}^n$ with $1$ at
position $i$, and $\mathbf{b}_j = \sum_i b_{ij} \mathbf{e}_i$ is the $j$th
column of $B$.
Then the Laurent expansion of any cluster variable,
with respect to the seed
$(\mathbf{x},\mathbf{y},B)$, is homogeneous with respect to this grading.
\end{proposition}

\begin{definition}
The \emph{$\gg$-vector} $\gg(x_\zg)$ of a cluster variable $x_\zg$,
with respect to the seed
$(\mathbf{x},\mathbf{y},B)$, is the multidegree 
of the Laurent expansion of $x_\zg$ with respect to 
$(\mathbf{x},\mathbf{y},B)$, using 
the $\gg$-vector grading of Proposition \ref{g}.
\end{definition}

\begin{remark}\label{rem g}
It follows from Proposition \ref{g} that the monomial in 
$x_i$'s and $y_j$'s whose exponent vector is the column
$\tilde{\mathbf{b}}_j$ of the extended $2n \times n$
matrix $\widetilde{B}$ has degree $0$.
\end{remark}
\begin{prop}\label{suffice}
Let  $\widetilde{B}$ be an $m \times n$ extended exchange matrix,
with linearly independent columns, and
let $\A = \A(\widetilde{B})$ be the associated cluster algebra,
with initial seed $(\{x_1,\dots,x_n\},\widetilde{B})$,
and coefficient variables $x_{n+1},\dots,x_m$.  
Let $U$ be a set of elements in $\A(\widetilde{B})$,
whose Laurent expansions with respect to the initial seed 
all have the form 
$$\x^g + \lambda_h \sum_h \x^{g+h},$$
where $\x^a$ denotes $x_1^{a_1}\dots x_m^{a_m}$, 
$\lambda_h$ is a scalar, 
and each $h$ is a 
non-negative linear combination of columns of $\widetilde{B}$.
Suppose moreover that the vectors $g$ and $g'$ associated to 
two different elements of $U$ differ in at least one of the 
first $n$ coordinates.
Then the elements of $U$
are linearly independent over the ground ring of $\A$.
\end{prop}

The proof below comes from the arguments of 
\cite[Remark 7.11]{FZ4}.
\begin{proof}
Because the columns of $\widetilde{B}$ are linearly 
independent, we can define a partial order on $\Z^m$ by 
$u \leq v$ if and only if $v$ can be obtained from $u$ 
by adding a non-negative linear combination of columns of 
$\widetilde{B}$.  Applying this partial order to 
Laurent monomials in $\{x_1,\dots,x_m\}$, it follows
that each element $\x^g + \lambda_h \sum_h \x^{g+h}$ of $U$ has
leading term $\x^g$.  Moreover, all leading terms have pairwise distinct
exponent vectors, and 
even if we multiply each element of $U$ by 
an arbitrary monomial in coefficient variables $x_{n+1},\dots,x_m$,
the leading terms will still have pairwise distinct exponent vectors.
Therefore any linear combination of elements of $U$ which sums to $0$
must necessarily have all coefficients equal to $0$.
\end{proof}

%%%%%%%%%%%%%%%%%%%%%%%%%%%%
%%
%%  section
%%
%%%%%%%%%%%%%%%%%%%%%%%%%%%%%%%%

\subsection{Cluster algebras arising from 
    surfaces}\label{sect surfaces} 

We follow the work of Fock and Goncharov \cite{FG1, FG2}, 
Gekhtman, Shapiro and Vainshtein \cite{GSV}, and
Fomin, Shapiro and Thurston \cite{FST}, who associated a cluster algebra
to any {\it bordered surface with marked points}. 
 In 
this section we will recall that construction in the special case of surfaces without punctures.

\begin{Def}
Let $S$ be a connected oriented 2-dimensional Riemann surface with
nonempty
boundary, and let $M$ be a nonempty finite subset of the boundary of $S$, such that each boundary component of $S$ contains at least one point of $M$. The elements of $M$ are called {\it marked points}. The
pair $(S,M)$ is called a \emph{bordered surface with marked points}.
\end{Def}
 
For technical reasons, we require that $(S,M)$ is not
a disk with 1,2 or 3 marked points.

\begin{definition}
An \emph{arc} $\zg$ in $(S,M)$ is a curve in $S$, considered up
to isotopy, such that: 
\begin{itemize}
\item[(a)] 
the endpoints of $\zg$ are in $M$;
\item[(b)] 
$\zg$ does not cross itself, except that its endpoints may coincide;
\item[(c)] 
except for the endpoints, $\zg$ is disjoint from   the boundary of $S$; and
\item[(d)] 
$\zg$ does not cut out a monogon or a bigon. 
\end{itemize}   
\end{definition}     

Curves that connect two
marked points and lie entirely on the boundary of $S$ without passing
through a third marked point are \emph{boundary segments}.
Note that boundary segments are not   arcs.

\begin{Def}
[\emph{Crossing numbers and compatibility of ordinary arcs}]
For any two arcs $\zg,\zg'$ in $S$, let $e(\zg,\zg')$ be the minimal
number of crossings of 
arcs $\za$ and $\za'$, where $\za$ 
and $\za'$ range over all arcs isotopic to 
$\zg$ and $\zg'$, respectively.
We say that arcs $\zg$ and $\zg'$ are  \emph{compatible} if $e(\zg,\zg')=0$. 
\end{Def}

\begin{Def}
A \emph{triangulation} is a maximal collection of
pairwise compatible arcs (together with all boundary segments). 
\end{Def}

\begin{Def}
Triangulations are connected to each other by sequences of 
{\it flips}.  Each flip replaces a single arc $\gamma$ 
in a triangulation $T$ by a (unique) arc $\gamma' \neq \gamma$
that, together with the remaining arcs in $T$, forms a new 
triangulation.
\end{Def}

\begin{Def}
Choose any   triangulation
$T$ of $(S,M)$, and let $\tau_1,\tau_2,\ldots,\tau_n$ be the $n$ arcs of
$T$.
For any triangle $\Delta$ in $T$, we define a matrix 
$B^\Delta=(b^\Delta_{ij})_{1\le i\le n, 1\le j\le n}$  as follows.
\begin{itemize}
\item $b_{ij}^\Delta=1$ and $b_{ji}^{\Delta}=-1$ if $\tau_i$ and $\tau_j$ are sides of 
  $\Delta$ with  $\tau_j$ following $\tau_i$  in the 
  clockwise order.
\item $b_{ij}^\Delta=0$ otherwise.
\end{itemize}
 
Then define the matrix 
$ B_{T}=(b_{ij})_{1\le i\le n, 1\le j\le n}$  by
$b_{ij}=\sum_\Delta b_{ij}^\Delta$, where the sum is taken over all
triangles in $T$.
\end{Def}

Note that  $B_{T}$ is skew-symmetric and each entry  $b_{ij}$ is either
$0,\pm 1$, or $\pm 2$, since every arc $\tau$ is in at most two triangles.

\begin{theorem} \cite[Theorem 7.11]{FST} and \cite[Theorem 5.1]{FT}
\label{clust-surface}
Fix a bordered surface $(S,M)$ and let $\Acal$ be the cluster algebra associated to
the signed adjacency matrix of a   triangulation. Then the (unlabeled) seeds $\Sigma_{T}$ of $\Acal$ are in bijection
with  the triangulations $T$ of $(S,M)$, and
the cluster variables are  in bijection
with the arcs of $(S,M)$ (so we can denote each by
$x_{\gamma}$, where $\gamma$ is an arc). Moreover, each seed in $\Acal$ is uniquely determined by its cluster.  Furthermore,
if a   triangulation $T'$ is obtained from another
  triangulation $T$ by flipping an arc $\gamma\in T$
and obtaining $\gamma'$,
then $\Sigma_{T'}$ is obtained from $\Sigma_{T}$ by the seed mutation
replacing $x_{\gamma}$ by $x_{\gamma'}$.
\end{theorem}

The exchange relation corresponding to a flip in a  
triangulation
is called 
a {\it generalized Ptolemy relation}.  It can be described as 
follows.
\begin{prop}\cite{FT}\label{Ptolemy}
Let $\alpha, \beta, \gamma, \delta$ be arcs 
or boundary segments 
of $(S,M)$ which cut out a quadrilateral; we assume that the sides
of the quadrilateral, listed in cyclic order, are
$\alpha, \beta, \gamma, \delta$.  Let $\eta$ and $\theta$ 
be the two diagonals of this quadrilateral; see the left of  
Figure \ref{figflip}.
Then 
\begin{equation} \label{shear-exchange} x_{\eta} x_{\theta} = Y x_{\alpha} x_{\gamma} + Y' x_{\beta} x_{\delta}\end{equation}
for some coefficients $Y$ and $Y'$.
\end{prop}

\begin{proof}
This follows from the interpretation of cluster variables as 
{\it lambda lengths}  and the 
Ptolemy relations for lambda lengths \cite[Theorem 7.5 and Proposition 6.5]{FT}.
\end{proof}

\begin{figure}
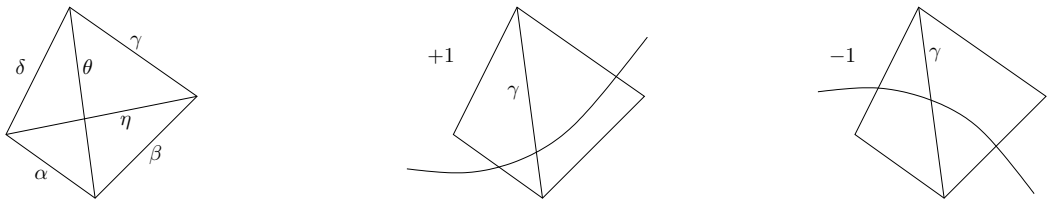
 \begin{center}
\scalebox{.8}{\input{figflip.pstex_t}} \hspace{6em} \scalebox{.8}{\input{figsz.pstex_t}}
\end{center}
\caption{
Exchange relation and shear coordinates}
\label{figflip}
\end{figure}

\subsubsection{Keeping track of coefficients using laminations}

So far we have not addressed the topic of 
coefficients for cluster algebras arising from bordered surfaces.
It turns out that 
W. Thurston's theory of measured laminations \cite{Shear} gives a 
concrete way to think about coefficients, 
as described in \cite[Sections 11-12]{FT} (see also \cite{FG3}).

\begin{Def}
A {\it lamination} on a bordered surface $(S,M)$ is a finite collection
of non-self-intersecting and pairwise non-intersecting curves in $S$,
modulo isotopy relative to $M$, subject to the following restrictions.
Each curve must be one of the following:
\begin{itemize}
\item a closed curve;
\item a curve connecting two unmarked points on the boundary of $S$.
\end{itemize}
Also, we forbid curves  with two endpoints on the boundary of $S$ which
are isotopic to a piece of boundary containing zero or one
 marked point.
\end{Def}

\begin{Def}
\label{def:shear}
Let $L$ be a lamination, and let $T$ be a triangulation.
For each arc $\gamma \in T$, 
the corresponding {\it shear coordinate} of $L$ with respect to $T$, 
denoted by $b_{\gamma}(T,L)$, is defined as a sum of contributions
from all intersections of curves in $L$ with $\gamma$.
Specifically, such an intersection contributes $+1$ (resp., $-1$)
to $b_{\gamma}(T,L)$ if the corresponding segment of a curve in 
$L$ cuts through the quadrilateral surrounding $\gamma$
as shown in Figure \ref{figflip} in the middle (resp., right).
\end{Def}

\begin{Def}
A {\it multi-lamination} is a finite family of laminations.
For any multi-lamination $\mathbf{L}=(L_{n+1}, \dots, L_{n+m})$  
and any triangulation $T$ of $(S,M)$,
define the matrix $\tilde{B} = \tilde{B}_{T,\mathbf{L}} = (b_{ij})$
as follows.  The top $n \times n$ part of $\tilde{B}$ is the signed
adjacency matrix $B_T$, with rows and columns indexed by arcs
$\gamma \in T$.
The bottom $m$ rows
are formed by the shear coordinates of the laminations $L_i$
with respect to $T$:
$$b_{n+i,\gamma} = b_{\gamma}(T, L_{n+i}) \text{ if }1 \leq i \leq m.$$
\end{Def}

By  \cite[Theorem 11.6]{FT},
the matrices
$\tilde{B}_{T,L}$ transform compatibly with mutation.

\begin{figure} \begin{center}
\input{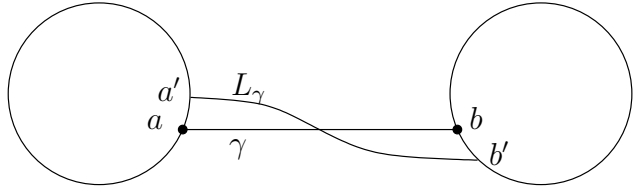}
\end{center}
\caption{Elementary lamination $L_\gamma$ corresponding to $\gamma$}
\label{laminate}
\end{figure}

\begin{Def}
\label{Li}
Let $\gamma$ be an arc in $(S,M)$.  Denote by 
$L_{\gamma}$ a lamination consisting of a single curve
defined as follows.  The curve $L_{\gamma}$ runs along 
$\gamma$ within a small neighborhood of it.  If $\gamma$
has an endpoint $a$ on a (circular) component $C$ of the boundary of $S$,
then $L_{\gamma}$ begins at a point $a'\in C$ located near $a$ in 
the counterclockwise direction, and proceeds along $\gamma$
as shown in Figure \ref{laminate}.
If $T$ is a triangulation, we let 
$L_T = (L_{\gamma})_{\gamma \in T}$ be the multi-lamination 
consisting of elementary laminations associated with 
the   arcs in $T$, and we call it 
the \emph{multi-lamination associated with $T$}.
\end{Def}

The following result comes from \cite[Proposition 16.3]{FT}.

\begin{prop}
Let $T$ be a   triangulation with   signed adjacency matrix
$B_T$. 
Let $L_T = (L_{\gamma})_{\gamma \in T}$
be the multi-lamination associated with $T$.
Then $\mathcal A(\tilde{B}_{T, L_T})$ is isomorphic to the cluster algebra with principal coefficients with respect to the matrix $B_T$, that is
$\Aprin(B_T) \cong
\mathcal A(\tilde{B}_{T, L_T})$.
\end{prop}

\subsection{Skein relations}
In this section we review some results from \cite{MW}.
\begin{definition}   \label{gen-arc}
A \emph{generalized arc}  in $(S,M)$ is a curve $\gamma$ in $S$ such that:
\begin{itemize}
\item[(a)] the endpoints of $\gamma$ are in $M$; 
\item[(b)] except for the endpoints,
$\gamma$ is disjoint  from the boundary of $S$; and
\item[(c)]
$\gamma$ does not cut out a monogon or a bigon.
\end{itemize}
\end{definition}

Note that we allow a generalized arc to cross itself a finite 
number of times.  We consider generalized 
arcs up to isotopy (of immersed
arcs).  In particular, an isotopy cannot remove a contractible
kink from a generalized arc.

\begin{definition}
A \emph{closed loop} in $(S,M)$ is a closed curve
$\gamma$ in $S$ which is disjoint from the 
boundary of $S$.  We allow a closed loop to have a finite
number of self-crossings.
As in Definition \ref{gen-arc}, we consider closed
loops up to isotopy.
\end{definition}

\begin{definition}
A closed loop in $(S,M)$ is called \emph{essential} if
  it is not contractible
and it does not have self-crossings.
\end{definition}

\begin{definition} (Multicurve)
We define a \emph{multicurve} to be a finite multiset of generalized  
arcs and closed loops such that there are only a finite number of pairwise crossings among the collection.
We say that a multicurve is \emph{simple}
if there are no pairwise crossings among the collection and 
no self-crossings.
\end{definition}

If a multicurve is not simple, 
then there are two ways to \emph{resolve} a crossing to obtain a multicurve that no longer contains this crossing and has no additional crossings.  This process is known as \emph{smoothing}.

\begin{definition}\label{def:smoothing} (Smoothing) Let $\gamma, \gamma_1$ and $\gamma_2$ be generalized  
arcs or closed loops such that we have
one of the following two cases:

\begin{enumerate}
 \item $\gamma_1$ crosses $\gamma_2$ at a point $x$,
  \item $\gamma$ has a self-crossing at a point $x$.
\end{enumerate}

\noindent Then we let $C$ be the multicurve $\{\gamma_1,\gamma_2\}$ or $\{\gamma\}$ depending on which of the two cases we are in.  We define the \emph{smoothing of $C$ at the point $x$} to be the pair of multicurves $C_+ = \{\alpha_1,\alpha_2\}$ (resp. $\{\alpha\}$) and $C_- = \{\beta_1,\beta_2\}$ (resp. $\{\beta\}$).

Here, the multicurve $C_+$ (resp. $C_-$) is the same as $C$ except for the local change that replaces the crossing {\Large $\times$} with the pair of segments {\LARGE $~_\cap^{\cup}$}
(resp. {\large $\supset \subset$}).    
\end{definition}   
See Figures \ref{skeinfigure} and \ref{skein3figure} for the first case, and
Figure \ref{skein5figure} for the second case.
\begin{figure}
\input{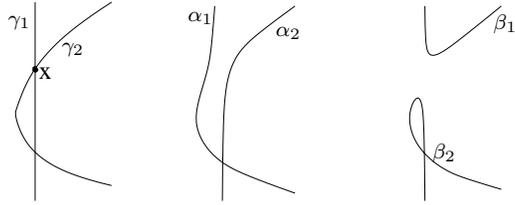}
\caption{Smoothing of two generalized arcs}\label{skeinfigure}
\end{figure}
\begin{figure}
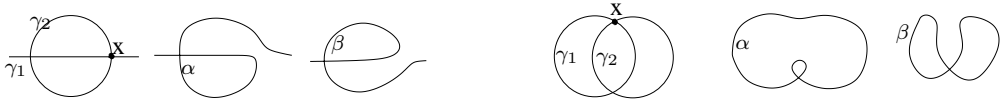

\input{Skein3.pstex_t} \hspace{1.4cm} \input{Skein4.pstex_t}
\caption{Smoothing of two curves where at least one is a loop}\label{skein3figure}
\end{figure}
\begin{figure}
\input{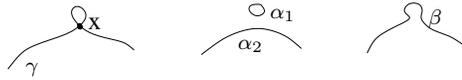}
\caption{Smoothing of a self-intersection}\label{skein5figure}
\end{figure}

Since a multicurve  may contain only a finite number of crossings, by repeatedly applying smoothings, we can associate to any multicurve  a collection of simple multicurves.  
  We call this resulting multiset of multicurves the \emph{smooth resolution} of the multicurve $C$.

\begin{theorem}\cite[Propositions 6.4,6.5,6.6]{MW}\label{th:skein1}
Let $C,$ $ C_{+}$, and $C_{-}$ be as in Definition \ref{def:smoothing}. 
Then we have the following identity in $\Aprin(B_T)$,
\begin{equation*}
x_C = \pm Y_1 x_{C_+} \pm Y_2 x_{C_-},
\end{equation*}
where $Y_1$ and $Y_2$ are monomials in the variables $y_{\tau_i}$.    
The monomials $Y_1$ and $Y_2$ can be expressed using the intersection numbers of the elementary laminations (associated to triangulation $T$) with  
the curves in $C,C_+$ and $C_-$.
\end{theorem}

\subsection{Chebyshev polynomials}\label{sect cheby}
{Chebyshev polynomials} will play an important role in the proof of our main result. In this section, we recall some basic facts.  
\begin{Def}
Let $T_k$ denote the $k$-th normalized Chebyshev polynomial with coefficients defined by \[T_k(t+\frac{Y}{t}) = t^k + \frac{Y^k}{t^k}.\]\end{Def}

\begin{prop} \label{chebyrec}
The normalized Chebyshev polynomials $T_k(x)$ defined above can also be uniquely determined by the initial conditions $T_0(x) = 2$, $T_1(x) = x$ and the recurrence $$T_k(x) = xT_{k-1}(x) - Y T_{k-2}(x).$$  If $Y$ is set to be $1$, then the $T_k(x)$'s can also be written as $2Cheb_k(x/2)$ where $Cheb_k(x)$ denotes the usual Chebyshev polynomial of the first kind, which satisfies 
$Cheb_k(\cos x) = \cos (kx)$. 
\end{prop}

\begin{proof}
It is easy to check that the unique one-parameter family of polynomials $T_k(x)$ defined by the property $T_k\left(t + \frac{Y}{t}\right) = t^k + \frac{Y^k}{t^k}$ satisfies the initial conditions $T_0(x) = 2$ and $T_1(x)=x$.  To see that this family also satisfies the desired recurrence,  we note that 
$$\left(t+ \frac{Y}{t}\right)\left(t^{k-1} + \frac{Y^{k-1}}{t^{k-1}}\right) = t^k + Y t^{k-2} + \frac{Y^{k-1}}{t^{k-2}} + \frac{Y^k}{t^k},$$ and thus letting $x=t + \frac{Y}{t}$, we obtain $$x T_{k-1}(x) = T_k(x) + Y T_{k-2}(x).$$  Since the usual Chebyshev polynomials satisfy the initial conditions $Cheb_0(x) = 1$, $Cheb_1(x) = x$, and the recurrence $$Cheb_k(x) = 2x Cheb_{k-1}(x) - Cheb_{k-2}(x),$$ the last remark follows as well.
\end{proof}

\begin{table}
\begin{eqnarray*}
T_0(x) &=& 2 \\ 
T_1(x) &=& x \\
T_2(x) &=& x^2-2Y \\
T_3(x) &=& x^3 - 3xY \\
T_4(x) &=& x^4 - 4x^2Y + 2Y^2 \\
T_5(x) &=& x^5 - 5x^3Y + 5xY^2 \\
T_6(x) &=& x^6 - 6x^4Y + 9x^2Y^2 - 2Y^3
\end{eqnarray*}
\caption{The normalized Chebyshev polynomials (with coefficients) $T_k(x)$ for small $k$.}
\end{table}

We record here one more   property of the normalized Chebyshev polynomials that we will need later.
\begin{prop} 
\label{prop cheb} For all $k\geq 1$, the monomial $x^k$ can be written as a positive linear combination of the normalized Chebyshev polynomials $T_k = T_k(x)$.  In particular,

{\scriptsize
\begin{eqnarray} 
\label{chebodd}  x^k &=& T_k + {k \choose 1} Y T_{k-2} 
+ \dots + {k \choose (k-1)/2 } Y^{(k-2)/2} T_1 \mathrm{~if~}k\mathrm{~is~odd,~and~} \\
\label{chebeven} x^k &=& T_k + {k \choose 1} Y T_{k-2} 
+ \dots + {k \choose (k-2)/2} Y^{(k-2)/2} T_2 + {k \choose k/2}Y^{k/2} \mathrm{~if~}k\mathrm{~is~even}.
\end{eqnarray}
}
\end{prop}

\begin{proof}
We prove both of these identities together by induction on $k$.  The base cases for $k=1$ or $2$ are easy to verify.  If $k \geq 3$ is odd, then by induction and equation (\ref{chebeven}), we obtain  
{\scriptsize
$$x^k = x(x^{k-1}) = x\left[ 
T_{k-1} + {k-1 \choose 1} Y T_{k-3} 
+ \dots + {k-1 \choose (k-3)/2} Y^{(k-3)/2}T_2 + {k -1\choose (k-1)/2} Y^{(k-1)/2}
\right].$$
}
The Chebyshev recurrence can be rewritten as $xT_{k-1} = T_k + Y T_{k-2}$.  Thus $x^k$ equals 
{\scriptsize
\begin{eqnarray*} 
&& \left[ 
T_{k} + {k-1 \choose 1} Y T_{k-2} 
+ {k-1 \choose 2} Y^2 T_{k-4} 
+ \dots + {k-1 \choose (k-3)/2} Y^{(k-3)/2} T_3\right] + {k -1\choose (k-1)/2} Y^{(k-1)/2} x
 \\
&+& Y
\left[T_{k-2} + {k-1 \choose 1} Y T_{k-4} 
+ {k-1 \choose 2} Y^2 T_{k-6} 
+ \dots + {k-1 \choose (k-3)/2} Y^{(k-3)/2}T_1
\right] \\ 
&=& 
T_{k} + {k \choose 1} Y T_{k-2} 
+ {k \choose 2} Y^2 T_{k-4} 
+ \dots + {k \choose (k-3)/2} Y^{(k-3)/2} T_3 + {k \choose (k-1)/2} Y^{(k-1)/2} T_1,
\end{eqnarray*}
}
where the last equality uses the fact that $x= T_1$.

A similar technique proves the identity for the case when $k$ is even, where we need to use the facts that $T_0 = 2$ and $2{ k-1 \choose (k-2)/2} = {k \choose k/2}$.  Using these and (\ref{chebodd}), the monomial $x^k = x(x^{k-1})$ equals
{\scriptsize
\begin{eqnarray*} 
&& \left[ 
T_{k} + {k-1 \choose 1} Y T_{k-2} 
+ {k-1 \choose 2} Y^2 T_{k-4} 
+ \dots + {k-1 \choose (k-4)/2} Y^{(k-4)/2} T_4 + {k -1\choose (k-2)/2} Y^{(k-2)/2} T_2\right]
 \\
&+& 
Y \left[T_{k-2} + {k-1 \choose 1} Y T_{k-4} 
+ {k-1 \choose 2} Y^2 T_{k-6} 
+ \dots + {k-1 \choose (k-4)/2} Y^{(k-4)/2} T_2 + {k-1 \choose (k-2)/2} Y^{(k-2)/2} T_0
\right] \\ 
&=& 
T_{k} + {k \choose 1} Y T_{k-2} 
+ {k \choose 2} Y^2 T_{k-4} 
+ \dots + {k \choose (k-2)/2} Y^{(k-2)/2} T_2 + {k \choose k/2} Y^{k/2}.
\end{eqnarray*}
} 

\end{proof}

%%%%%%%%%%%%%%%%%%%%%%%%%%%%%
%%
%% SECTION
%%
%%%%%%%%%%%%%%%%%%%%%%%%%%%%%%%

\section{Definition of the two bases $\B^\circ$ and $\B$}
\label{sec two bases}

Throughout Sections \ref{sec two bases}--\ref{full-rank} of this paper, we fix an unpunctured marked surface $(S,M)$ and a triangulation $T$,
and consider the corresponding cluster algebra 
$\A=\Aprin(B_T)$,
with principal coefficients with respect to $T$.
Recall that the cluster variables of $\A$ are in bijection with 
the arcs in $(S,M)$.  In this paper we will associate elements
of $\A$ to any \emph{generalized}   arc
(where self-intersections are allowed) and to any closed loop.
In particular, we will define two sets
$\C^{\circ}(S,M)$ and $\C(S,M)$ of collections of
loops and   arcs
in $(S,M)$, and will associate a cluster algebra element to each element of 
$\C^{\circ}(S,M)$ and $\C(S,M)$.

\subsection{Snake graphs and band graphs}\label{sect tiles}\label{sect graph}

Recall from \cite{MSW} that we have a positive combinatorial formula for the Laurent expansion of any cluster variable in a cluster algebra arising from a surface.  Each such cluster variable corresponds to an   arc in the surface, thus our formula associates a cluster algebra element to every   arc. We will generalize this construction and associate cluster algebra elements to \emph{generalized}   arcs,  as well as to closed loops (with or without selfcrossings).

 Let 
$\zg$ be an   arc in $(S,M)$ which is not in $T$. 
Choose an orientation on $\zg$, let $s\in M$ be its starting point, and let $t\in M$ be its endpoint.
We denote by
$s=p_0, p_1, p_2, \ldots, p_{d+1}=t$
the points of intersection of $\zg$ and $T$ in order.  
Let $\tau_{i_j}$ be the arc of $T$ containing $p_j$, and let 
$\zD_{j-1}$ and 
$\zD_{j}$ be the two   triangles in $T$ 
on either side of 
$\tau_{i_j}$. Note that each of these triangles has three distinct sides, but not necessarily three distinct vertices, see Figure \ref{figr1}.
\begin{figure}
\includegraphics{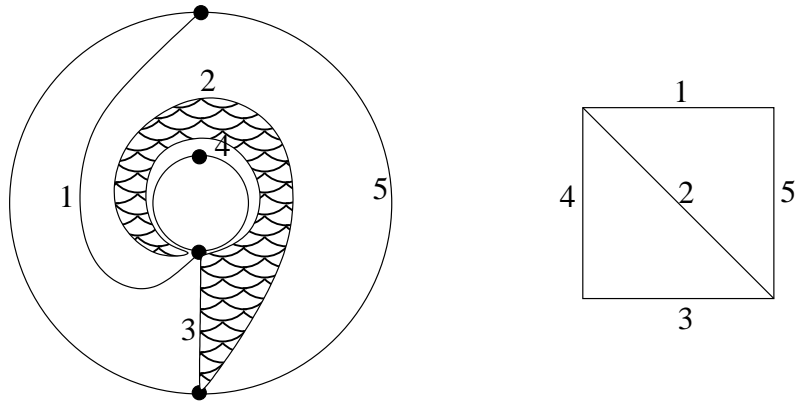}
\caption{On the left, a triangle with two vertices; on the right the tile $G_{j}$ where $i_j=2$. }\label{figr1}
\end{figure}
Let $G_j$ be the graph with 4 vertices and 5 edges, having the shape of a square with a diagonal, such that there is a bijection between the edges of $G_j$ and the 5 arcs in the two triangles $\zD_{j-1}$ and $\zD_j$, which preserves the signed adjacency of the arcs up to sign and such that the diagonal in $G_j$ corresponds to the arc $\tau_{i_j}$ containing the crossing point $p_j$. We call the graph $G_j$ a \emph{tile}.
Thus the tile $G_j$ is given by the quadrilateral in the triangulation $T$ whose diagonal is $\tau_{i_j}$. 

\begin{definition} \label{relorient}  
Given a planar embedding $\tilde G_j$ 
of a  tile $G_j$, we define the \emph{relative orientation} 
$\mathrm{rel}(\tilde G_j, T)$ 
of $\tilde G_j$ with respect to $T$ 
to be $\pm 1$, based on whether its triangles agree or disagree in orientation with those of $T$.  
\end{definition}
For example, in Figure \ref{figr1}, the tile $G_j$ has relative orientation $+1$.

Using the notation above, 
the arcs $\tau_{i_j}$ and $\tau_{i_{j+1}}$ form two edges of a triangle $\zD_j$ in $T$.  Define $\tau_{[\zg_j]}$ to be the third arc in this triangle.

We now recursively glue together the tiles $G_1,\dots,G_d$
in order from $1$ to $d$, so that for two adjacent   tiles, 
 we glue  $G_{j+1}$ to $\tilde G_j$ along the edge 
labeled $\tau_{[\zg_j]}$, choosing a planar embedding $\tilde G_{j+1}$ for $G_{j+1}$
so that $\mathrm{rel}(\tilde G_{j+1},T) \not= \mathrm{rel}(\tilde G_j,T).$  See Figure \ref{figglue}.

\begin{figure}
\input{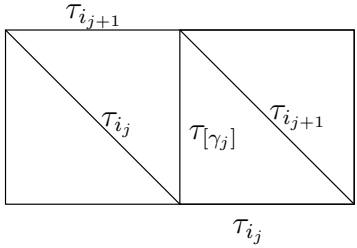}
\caption{Gluing tiles $\tilde G_j$ and $\tilde G_{j+1}$ along the edge labeled  $\tau_{[\zg_j]}$}
\label{figglue}
\end{figure}

After gluing together the $d$ tiles, we obtain a graph (embedded in 
the plane),
which we denote  by
$\overline{G}_{\zg}$. 
\begin{definition}
The \emph{snake graph} $G_{\gamma}$ associated to $\gamma$ is obtained 
from $\overline{G}_{\zg}$ by removing the diagonal in each tile.
\end{definition}

In Figure \ref{figsnake}, we give an example of an 
 arc $\gamma$ and the corresponding snake graph 
${G}_{\zg}$. Since $\gamma$ intersects $T$
five times,  
${G}_{\zg}$ has five tiles. 

\begin{figure}
\input{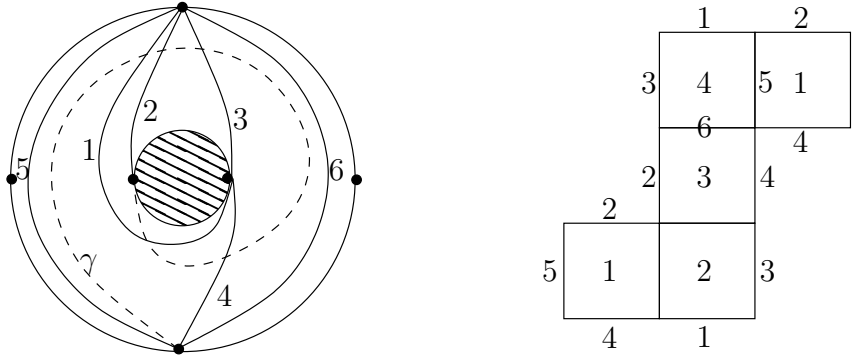}
\caption{An arc $\gamma$ in a triangulated annulus on the left and the corresponding snake graph $G_{\gamma}$ on the right. The tiles labeled 1,3,1 have positive relative orientation and the tiles 2,4 have negative relative orientation.}\label{figsnake}
\end{figure}

\begin{remark}
Even if $\gamma$ is a generalized arc, thus allowing self-crossings, we still can define 
  ${G}_{\gamma}$ in the same way.  
\end{remark}

Now we associate a similar graph to closed loops.
Let $\zeta$ be a closed loop in $(S,M)$, which may or may not have self-intersections, 
but which is not contractible and has no contractible kinks.  Choose
an orientation for $\zeta$, and a 
triangle $\Delta$ which is crossed by $\gamma$.  Let $p$ be a point in the interior of $\Delta$ which lies on $\gamma$, and let $b$ and $c$ be the two sides of the triangle crossed by $\gamma$
immediately before and following its travel through point $p$.
Let $a$ be the third side of $\Delta$.   
We let $\tilde{\gamma}$ denote the arc from $p$ back to itself that exactly
follows closed loop $\gamma$.  See the left of Figure \ref{fig band}.

\begin{figure}
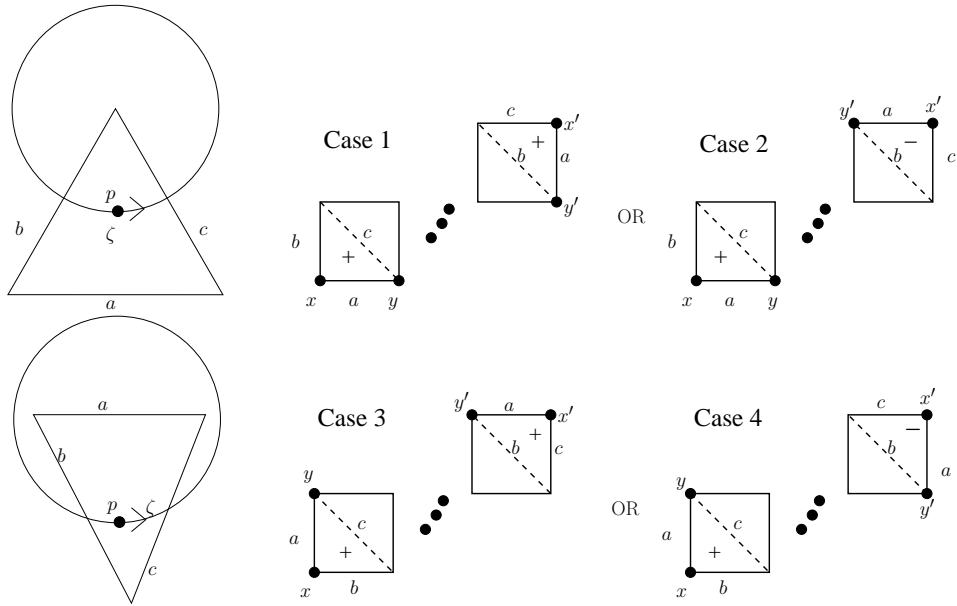

\scalebox{0.2}
{
\input{BandTriag.pstex_t}
\hspace{8em}
}
\scalebox{0.3}
{
\input{BandGs.pstex_t}
}
\scalebox{0.2}
{
\input{BandTriag2.pstex_t}
\hspace{8em}
}
\scalebox{0.3}
{
\input{BandGs2.pstex_t}
}
\caption{(Left): Triangle containing $p$ along closed loop $\zeta$. \hspace{10em}
(Right): Corresponding Band graph (with $x\sim x'$, $y\sim y'$) depending on whether $\gamma$ crosses an odd or even number of arcs. The $+$'s and $-$'s denote relative orientation of each tile}
\label{fig band}
\end{figure}

We start by building the snake graph $G_{\tilde{\gamma}}$ 
as defined above.  In the first tile of $G_{\tilde{\gamma}}$,  
let $x$ denote the vertex
at the corner of the edge labeled $a$ and the edge labeled $b$, and let $y$ denote the vertex at the other end of the edge labeled $a$.  Similarly, in the last tile of
$G_{\tilde{\gamma}}$,  
let $x'$ denote the vertex at the corner of the edge labeled
$a$ and the edge labeled $c$, and let $y'$ denote the vertex at the other
end of the edge labeled $a$.  See the right of Figure \ref{fig band}.

\begin{definition} \label{def band} The \emph{band graph}  $\widetilde{G}_{ \zeta}$ associated to the loop $\zeta$
is the graph obtained from $G_{\tilde{\zeta}}$ by identifying the edges labeled $a$ in the first and last tiles so that the vertices $x$ and $x'$ and the vertices $y$ and $y'$ are glued together.  We refer to the two vertices obtained by identification
as $x$ and $y$, and to the edge obtained by identification as the \emph{cut edge}.
The resulting graph lies on an annulus or a M\"obius strip.  
\end{definition}

\subsection{Laurent polynomials associated to generalized   arcs and closed loops}\label{sect def loops}

Recall that if $\tau$ is a boundary segment then $x_{\tau} = 1$,

\begin{definition}  
If $\zg$ is a generalized  arc or a closed loop and  $\tau_{i_1}, \tau_{i_2},\dots, \tau_{i_d}$
is the sequence of arcs in $T$ which $\zg$ crosses, we define the \emph{crossing monomial} 
of $\gamma$ with respect to $T$ to be
$$\mathrm{cross}(T, \gamma) = \prod_{j=1}^d x_{\tau_{i_j}}.$$
\end{definition}

\begin{definition} 
A \emph{perfect matching} of a graph $G$ is a subset $P$ of the 
edges of $G$ such that
each vertex of $G$ is incident to exactly one edge of $P$. 
If $G$ is a snake graph or a band graph, and the edges of  a perfect matching $P$ of  
$G $ are labeled $\tau_{j_1},\dots,\tau_{j_r}$, then 
we define the {\it weight} $x(P)$ of $P$ to be 
$x_{\tau_{j_1}} \dots x_{\tau_{j_r}}$.
\end{definition}

\begin{definition}  Let $\zg$ be a generalized arc.
It is easy to see that the snake graph
$G_{\zg}$  
has  precisely two perfect matchings which we call
the {\it minimal matching} $P_-=P_-(G_{\zg})$ and 
the {\it maximal matching} $P_+
=P_+(G_{\zg})$, 
which contain only boundary edges.
To distinguish them, 
if  $\mathrm{rel}(\tilde G_1,T)=1$ (respectively, $-1$),
we define 
$e_1$ and $e_2$ to be the two edges of 
$\overline{G}_{\zg}$ which lie in the counterclockwise 
(respectively, clockwise) direction from 
the diagonal of $\tilde G_1$.  Then  $P_-$ is defined as
the unique matching which contains only boundary 
edges and does not contain edges $e_1$ or $e_2$.  $P_+$
is the other matching with only boundary edges.
\end{definition}
In the example of Figure \ref{figsnake}, the minimal matching $P_-$ contains the bottom edge of the first tile labeled 4.

\begin{definition} 
\label{def good matching} Let $\zeta $ be a closed loop. 
A  perfect matching $P$ of the band graph $\widetilde{G}_\zeta$ is called a \emph{good matching} 
if either $x$ and $y$ are matched to each other ($P(x)=y$ and $P(y)=x$) or if both 
edges $(x,P(x))$ and $(y,P(y))$ lie on one side of the cut edge.
\end{definition}

\begin{remark}\label{descend}
Let $\widetilde{G}_\zeta$ be a band graph obtained by identifying two edges
of the snake graph $G_{\tilde\zeta}$.  The good matchings of
$\widetilde{G}_\zeta$ can be identified with a subset of the perfect matchings
of $G_{\tilde\zeta}$.  Let $\widetilde{P}$ be a good matching of $\widetilde{G}_\zeta$.
Thinking of $\widetilde{P}$ as a subset of edges of $G_{\tilde\zeta}$, then
by definition of good we can add to it either the edge $(x,y)$ or
the edge $(x',y')$ to get a perfect matching $P$ of $G_{\tilde\zeta}$.  In this case,
we say that the perfect matching $P$ of $G_{\tilde\zeta}$ \emph{descends} to a
good matching $\widetilde{P}$ of $\widetilde{G}_\zeta$.  In particular,
the minimal matching $P_-$ of $G_{\tilde\zeta}$ descends to a good matching of
$\widetilde{G}_\zeta$, which we also call \emph{minimal}.
(To see this, just consider the cases of whether $G_{\tilde\zeta}$ has an odd or
even number of tiles, and observe that the minimal matching
of $G_{\tilde\zeta}$ always uses one of the edges $(x,y)$ and $(x',y')$.)
\end{remark}

For an arbitrary perfect matching $P$ of a snake graph $G_{\gamma}$, we let
$P_-\ominus P$ denote the symmetric difference, defined as $P_-\ominus P =(P_-\cup P)\setminus (P_-\cap P)$.

\begin{lemma}\cite[Theorem 5.1]{MS}
\label{thm y}
The set $P_-\ominus P$ is the set of boundary edges of a 
(possibly disconnected) subgraph $G_P$ of $G_\zg$, 
which is a union of cycles.  These cycles enclose a set of tiles 
$\cup_{j\in J} G_{j}$,  where $J$ is a finite index set.
\end{lemma}

We use this decomposition to define \emph{height monomials} for
perfect matchings.
Note that the exponents in the height monomials defined below coincide
with the definition of height functions given in
\cite{ProppLattice} for perfect matchings of bipartite
graphs, based on earlier work of \cite{ConwayLagarias}, 
\cite{EKLP}, and \cite{Thurston} for domino tilings.

\begin{definition} \label{height} With the notation of Lemma \ref{thm y},
we  define the \emph{height monomial} $y(P)$ of a perfect matching $P$ of a snake graph $G_\gamma$ by
\begin{equation*}
y(P) = \prod_{j\in J} y_{\tau_{i_j}}.
\end{equation*}
The \emph{height monomial $y(\tilde P)$ of a good matching} $\tilde P$ of a band graph $\widetilde{G}_\zeta$ is defined to be the height monomial of the corresponding matching on the snake graph $G_{\tilde\zeta}$.

\end{definition}

For each generalized arc $\gamma$, we now define a Laurent polynomial $x_\zg$, as well as a polynomial 
$F_\zg^T$ obtained from $x_{\gamma}$ by specialization.
\begin{definition} \label{def:matching}
Let $\zg$ 
be a generalized arc and let $G_\zg$,
be its snake graph.  
\begin{enumerate}
\item If $\zg$ 
has a contractible kink, let $\overline{\zg}$ denote the 
corresponding generalized arc with this kink removed, and define 
$x_{\zg} = (-1) x_{\overline{\zg}}$.  
\item Otherwise, define
\[ x_{\gamma}= \frac{1}{\mathrm{cross}(T,\zg)} \sum_P 
x(P) y(P),\]
 where the sum is over all perfect matchings $P$ of $G_{\zg}$.
\end{enumerate}
Define $F_{\gamma}^T$ to be the polynomial obtained from 
$x_{\gamma}$ by specializing all the $x_{\tau_i}$ to $1$.

If $\gamma$ is a curve that 
cuts out a contractible monogon, then we define $x_\gamma =0$.     
\end{definition}

\begin{theorem}\cite[Theorem 4.9]{MSW}
\label{thm MSW}
If $\gamma$ is an arc, then 
$x_{\gamma}$ 
is   the Laurent expansion with respect to the the seed $\Sigma_T$ of the cluster variable in $\A$ corresponding to the arc $\gamma$, and $F_{\gamma}^T$ is its \emph{F-polynomial}.
\end{theorem}

For every closed loop $\zeta$, we now define a Laurent polynomial $x_\zeta$, as well 
as a polynomial 
$F_\zeta^T$ obtained from $x_{\zeta}$ by specialization.
\begin{definition} \label{def closed loop} 
Let $\zeta$  
be a closed loop.  
\begin{enumerate}
\item If $\zeta$ is a contractible loop,
 then let $x_\zeta = -2$.
\item If $\zeta$ 
has a contractible kink, let $\overline{\zeta}$ denote the 
corresponding closed loop with this kink removed, and define 
$x_{\zeta} = (-1) x_{\overline{\zeta}}$.
\item Otherwise, let 
$$x_{\zeta} = \frac{1}{\mathrm{cross}(T,\zg)} \sum_{P} 
x(P) y(P),$$ where the sum is over all good matchings $P$ of the  band graph $\widetilde{G}_{\zeta}$.
\end{enumerate}
Define $F_\zeta^T$ to be the Laurent polynomial obtained from 
$x_\zeta$ by specializing all the $x_{\tau_i}$ to $1$.
\end{definition}

\begin{remark}\label{remnot}
Note that $x_{\gamma}$ depends on the triangulation $T$ and the surface $(S,M)$, and lies in
     (the fraction field of) $\Aprin(B_T).$  If we want to emphasize the dependence on $T$, we
     will use the notation $X_{\gamma}^T$ instead of $x_{\gamma}$. Similarly for  $X_{\zeta}^T$ and $x_{\zeta}$.
     \end{remark}

\subsection{Bangles and bracelets}\label{sec:bangbrac}

\begin{definition}
Let $\zeta$ be an essential loop in $(S,M)$.  
We define the 
\emph{bangle} $\Bang_k \zeta$ to be the union of $k$ loops isotopic to $\zeta$.
(Note that $\Bang_k \zeta$ has no self-crossings.)  And we define the 
\emph{bracelet} $\Brac_k \zeta$ to be the closed loop obtained by 
concatenating $\zeta$ exactly $k$ times, see Figure \ref{figbangbrac}.  (Note that it will have $k-1$ self-crossings.)
\end{definition}

\begin{figure}
\includegraphics{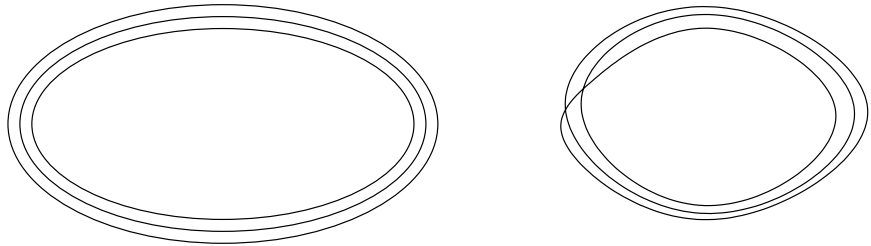}
\caption{A bangle $\Bang_3 \zeta$, on the left, and a bracelet $\Brac_3 \zeta$, on the right.}\label{figbangbrac}
\end{figure}
Note that $\Bang_1 \zeta = \Brac_1 \zeta = \zeta$.

\begin{definition}
\label{def C0-compatible}
A collection $C$ of arcs and essential loops is called 
\emph{$\C^{\circ}$-compatible} if no two elements of $C$ cross each other.
We define $\C^{\circ}(S,M)$  to be the set of all 
$\C^{\circ}$-compatible collections in $(S,M)$.
\end{definition}

\begin{definition}
A collection $C$ of arcs and bracelets is called 
\emph{$\C$-compatible} if:
\begin{itemize}
\item no two elements of $C$ cross each other except for the self-crossings of a bracelet; and
\item given an essential loop $\zeta$ in $(S,M)$, 
there is at most one $k\ge 1$ such
that the $k$-th bracelet $\Brac_k\zeta$ lies in $C$, and, moreover, there is at
most one copy of this bracelet $\Brac_k\zeta$ in $C$.
\end{itemize}
We define $\C(S,M)$ to be the set of all $\C$-compatible
collections in $(S,M)$.
\end{definition}

Note that a $\C^{\circ}$-compatible collection may contain 
bangles $\Bang_k \zeta$ for $k \geq 1$, but it will not contain
bracelets $\Brac_k \zeta$ except when $k=1$.
And  a $\C$-compatible collection may contain bracelets, but will never
contain a bangle $\Bang_k \zeta$ except when $k=1$.

\begin{definition}
Given an arc or closed loop $c$, let 
$x_c$ denote the corresponding Laurent polynomial
defined in Section \ref{sect def loops}.
We define $\B^\circ$ 
to be the set of all cluster algebra 
elements in $\A = \Aprin(B_T)$ corresponding to the set $C^{\circ}(S,M)$, 
\[\B^{\circ} = \left\{\prod_{c\in C} x_c \ \vert \ C \in \C^{\circ}(S,M) \right\}.\]
Similarly, we define
\[\B = \left\{\prod_{c\in C} x_c \ \vert \ C \in \C(S,M) \right\}.\]
\end{definition}

\begin{remark}
Both $\B^{\circ} $ and $\B $ contain
the cluster monomials of $\A $.
\end{remark}

\begin{remark} The notation $\C^{\circ}$ 
is meant to remind the reader that this collection includes bangles.
We chose to use the unadorned notation $\C$ for the other collection
of arcs and loops, because the corresponding set $\B$ of cluster algebra
elements is believed to have better positivity properties than
the set $\B^{\circ}$.
\end{remark}

%%%%%%%%%
%%%%%%%%%%
%% 
%% SECTION
%%
%%%%%%%%%%%%
%%%%%%%%%%

\section{Proofs of the main result}
\label{sec:main}

The goal of this section is to prove that both sets $\B^\circ$ and $\B$ are bases for the cluster algebra $\A$. 
More specifically, we will prove the following.
\begin{theorem}\label{thm main} If the surface has no punctures and at least two marked points then
the sets $\B^\circ$ and $\B$ are bases of the cluster algebra $\A$.                             
\end{theorem}

We subdivide the proof into the following three steps:
\begin{enumerate}
\item  $\B^\circ$ and $\B$ are subsets of $\A$.
\item  $\B^\circ$ and $\B$ are spanning sets for $\A$.
\item  $\B^\circ$ and $\B$ are linearly independent.
\end{enumerate}

\subsection{$\B^\circ$ and $\B$ are subsets of $\A$}
We start by describing the relation between bangles and bracelets involving the Chebyshev polynomials.

If $\tau, \zeta$ are arcs or closed loops and $L$ is a lamination, we let $e(\tau,\zeta)$ (resp. $e(\tau,L)$) denote the number of crossings between $\tau$ and $\zeta$ (resp. $L$.).

\begin{proposition} \label{prop Cheby} Let $\zeta$ be an essential loop, and  let $Y_\zeta = \prod_{\tau \in T} y_\tau^{e(\zeta,\tau)}$. Then we have
$$ x_{\Brac_k \zeta} =
T_k(x_\zeta),$$ 
where $T_k$ denotes the $k$th normalized Chebyshev polynomial (with coefficients) defined in Section \ref{sect cheby}.
\end{proposition}

\begin{proof} We prove the statement by induction on $k$. Smoothing $\Brac_{k+1}\zeta$ at one point of self-crossing produces the multicurves $\{\zeta,\Brac_k\zeta \}$ and $\{\gamma\}$, where $\zg$ is the curve $\Brac_{k-1}$ with a contractible kink.
It follows from Theorem \ref{th:skein1} that
\[x_{\Brac_{k+1}\zeta} = \pm x_\zeta x_{\Brac_k \zeta} \prod_{i=1}^n y_i^{(c_i-a_i)/2} 
\pm x_{\Brac_{k-1}\zeta} \prod_{i=1}^n y_i^{(c_i-b_i)/2},\]
where $c_i=e(\Brac_{k+1}\zeta,L_i)$,
$a_i=e(\Brac_{k}\zeta,L_i)+e(\zeta,L_i)$ and
$b_i=e(\Brac_{k-1}\zeta,L_i)$.
From the definition of bracelets, it follows that $c_i=a_i$ and that $c_i=b_i+2e(\zeta,\tau_i)$. Thus
\[x_{\Brac_{k+1}\zeta} = \pm x_\zeta x_{\Brac_k \zeta} \pm x_{\Brac_{k-1}\zeta} Y_\zeta.\]
It remains to show that the first sign is $+$ and the second is $-$.

Since $k\ge 1$, each   of $x_\zeta, x_{\Brac_k \zeta}$ and $x_{\Brac_{k+1} \zeta}$ is  a Laurent polynomial given by a band graph formula.
So in particular, each is in 
$\Z[x_i^{\pm 1}, y_i],$
 has all signs positive, and has a unique term without any coefficients $y_i$, corresponding to the minimal matching. On the other hand, $Y_\zeta$ is a monomial in the $y_i$'s which is not equal to $1$.
If we set all the $x_i$'s equal to 1, and the $y_i$'s equal to 0,
then we get $1 = \pm1\pm 0$, which shows that the first sign is $+$.

To see that the second sign is $-$, we use Definition \ref{def closed loop} and the specialization
$x_i=1$ and $y_i=1$ for all $i$.  Letting 
$\Good(G)$ denote the set of good matchings of $G$, and letting $\tilde{G}_{m\zeta}$ be a shorthand for the band graph $\tilde{G}_{\Brac_m\zeta}$, our equation becomes 
\[|\Good(\tilde{G}_{(k+1)\zeta})|=+|\Good(\tilde{G}_\zeta)|\cdot|\Good(\tilde{G}_{k\zeta})|\pm |\Good(\tilde{G}_{(k-1)\zeta})|.\]
It thus suffices to show that
$$|\Good(\tilde{G}_{(k+1)\zeta})| <
|\Good(\tilde{G}_{\zeta})| \cdot |\Good(\tilde{G}_{k\zeta})|.$$ 

For $d\geq 2$, we let $\bullet_{y'}\line(1,0){25}\bullet_{x'}$ denote the
edge of the snake graph $G_{d\zeta}$ or the band graph $\tilde{G}_{d\zeta}$ succeeding the last tile of the subgraph $G_\zeta$.  We will exhibit an injective map
$\psi: \Good(\tilde{G}_{(k+1)\zeta}) \longrightarrow \Good(\tilde{G}_{\zeta}) \times  \Good(\tilde{G}_{k\zeta}).$
In particular, given $\tilde{P} \in \Good(\tilde{G}_{(k+1)\zeta})$, we define $\psi(\tilde{P}) = (\tilde{Q}_1,\tilde{Q}_2)$ by the following:
\begin{itemize}
\item {\bf Lift} $\tilde{P}$ to $P$, a perfect matching of the snake graph $G_{(k+1)\zeta}$.
\item{\bf Split} $P$ along the edge $\bullet_{y'}\line(1,0){25}\bullet_{x'}$ into perfect matchings $P_1$ and $P_2$ of the snake graphs $G_\zeta$ and $G_{k\zeta}$, respectively.  Note there are
two cases here.  If the edge $\bullet_{y'} \line(1,0){25}\bullet_{x'}$  is in $ P$, we copy it, and include
it as a distinguished edge in both $P_1$ and $P_2$.
Otherwise, either $P_1$ or $P_2$ is missing one edge
to be a perfect matching, and we adjoin the edge $\bullet_{y'} \line(1,0){25}\bullet_{x'}$ 
to that  perfect matching.

\item {\bf Swap}.
Consider the symmetric difference $P_1 \ominus P_2$ which, by Lemma \ref{thm y},
consists of a union of cycles, and let $C$ be the cycle which encloses
the
tile $G_1$, if such a cycle exists, and let $C$ be empty otherwise. We
then
define the  \emph{first segment} of both $P_1$ and $P_2$ to be the matching on the
induced subgraph formed by the tiles enclosed by the cycle $C$. Swap the first segments
 of $P_1$ and $P_2$ to obtain new perfect matchings of $G_\zeta$ and $G_{k\zeta}$, which we denote as $Q_1$ and $Q_2$.  
\item {\bf Descend} $Q_1$ and $Q_2$ down to good matchings $\tilde{Q}_1$ and $\tilde{Q}_2$ of the band graphs $\tilde{G}_\zeta$ and $\tilde{G}_{k\zeta}$.
\end{itemize}

A straightforward analysis of nine possible cases (contingent on how the perfect matching $P$ looks locally around edges $\bullet_{x}\line(1,0){25}\bullet_{y}$ and 
		$\bullet_{y'}\line(1,0){25}\bullet_{x'}$) shows that the map $\psi$ is well-defined and has a left-inverse.  In particular, swapping the first segments of $P_1$ and $P_2$ turns the condition that $\tilde{P}$ is a \emph{good} matching of the band graph $\tilde{G}_{(k+1)\zeta}$ into the condition that $\tilde{Q_1}$ and $\tilde{Q_2}$ are \emph{good} matchings of the band graphs $\tilde{G}_\zeta$ and $\tilde{G}_{k\zeta}$.
\end{proof}

\begin{remark}
In the special case where the cluster algebra $\A$ has trivial coefficients, a similar formula can be found in \cite{FroGel}.
\end{remark}
\begin{remark}
In the special case where the surface is an annulus, Chebyshev polynomials were used in \cite{D,DuTh} to construct atomic basis for the cluster algebra.
\end{remark}

Next we show that the sets $\B^\circ$ and $\B$ are subsets of the cluster algebra, using our assumption that the number of marked points is at least 2. We do not know whether the result is true for surfaces with exactly one marked point. 

\begin{proposition}\label{lem:containment} If the surface has at least two marked points then 
the sets  $\B^\circ$ and $\B$ are subsets of $\A$.
\end{proposition}
\begin{proof}
First recall that if $\zg$ is an arc, then $x_\zg$ is a cluster variable by \cite{MSW}. Thus if $C$ is a multicurve consisting of non-crossing arcs, then $x_C$ is a monomial of cluster variables, hence $x_C\in \A$.

Next suppose that  $\zeta$ is an essential loop. Suppose first that there exists one boundary component which contains at least two marked points $m_1$ and $m_2$.
Let $\gamma$ be the arc obtained by attaching the loop $\zeta$ to the point $m_1$; more precisely, 
$\gamma$ is the isotopy class of the curve $\gamma_1\zeta\gamma_1^{-1}$, where $\gamma_1$ is a curve from $m_1$ to the starting point of $\zeta$, see Figure \ref{fig element}. Let $\gamma'$ be the unique arc that crosses $\zg$ twice, connects the two immediate neighbors $m_1^-$ and $m_1^+$ of $m_1$ on the boundary, and is homotopic to the part of the boundary component between $m_1^-$ and $m_1^+$. Note that $m_1^-$ and $m_1^+$ coincide if this boundary component contains exactly two marked points.
The multicurve $C=\{\gamma,\gamma'\}$ smoothes to the four simple multicurves shown in Figure \ref{fig element2}, and it follows from Theorem \ref{th:skein1} that  
\[x_\zg x_{\zg'}= 0 \pm y(\za:C) x_\za \pm y(\zb:C) x_\zb \pm y(\zeta:C) x_\zeta  
\]
for some coefficients $y(\za:C)$, $y(\zb:C)$ and $y(\zg:C)$.
Solving for $x_\zeta$, we get \[x_\zeta=\big(x_\zg x_{\zg'}\pm y(\za:C) x_\za \pm y(\zb:C) x_\zb\big)/ y(\zeta:C),\]
which shows that $x_\zeta\in \mathcal{A}$.
\begin{figure}
\scalebox{0.6}{\input{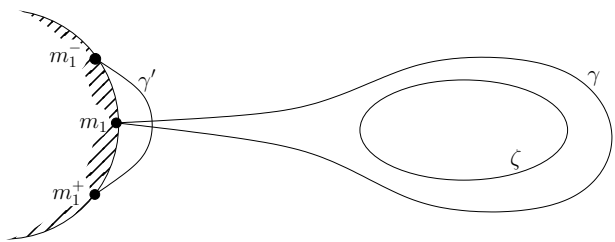}}
\caption{Two arcs $\zg,\zg'$ associated to the essential loop $\zeta$. The smoothing of the multicurve $\{\zg,\zg'\}$ is shown in Figure \ref{fig element2}.}
\label{fig element}
\end{figure}

\begin{figure}
\scalebox{0.6}{\input{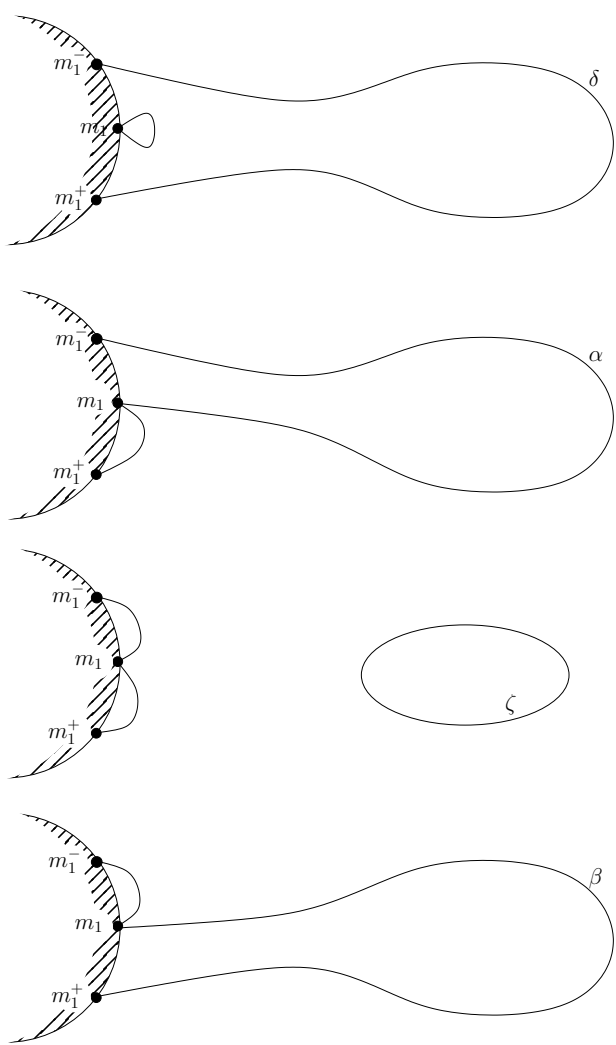}}
\caption{Smoothing of the multicurve $\{\zg,\zg'\}$ of Figure \ref{fig element}.}\label{fig element2}
\end{figure}

Now suppose  that each boundary component contains exactly one marked point. Then, by our assumption, there exist at least two such boundary components $D_1$ and $D_2$. Let $m_i$ denote the marked point on $D_i$.
Choose two distinct points $p_1$ and $p_2$ on the loop  $\zeta$, fix an orientation of $\zeta$,  and denote by $\zeta_1$ the segment of $\zeta$ from $p_1$ to $p_2$ and by $\zeta_2$ the segment of $\zeta$ from $p_2$ to $p_1$.
Let $\zg_1$ be a curve from $m_1 $ to $p_1$ and  $\zg_2$ a curve from $m_2 $ to $p_2$.
Define $\zg$ to be the arc homotopic to the concatenation $\zg_1\zeta_1\zg_2^{-1}$, see   Figure \ref{fig element3}.

To define $\gamma'$, we start with the arc from $m_1$ to $m_2$ given by $\zg_1\zeta_2^{-1}\zg_2^{-1}$ and add to it a complete lap around each of the boundary components $D_1,D_2$ in the directions that create crossings with $\zg$. In Figure \ref{fig element3}, $\zg'$ corresponds to the concatenation $\zd_1\zg_1\zeta_2^{-1}\zg_2^{-1}\zd_2$, where $\zd_i$ is a curve that starts and ends at $m_i$ and goes around the boundary component $D_i$ exactly once.

Then the multicurve $C=\{\gamma,\gamma'\}$ smoothes to the four simple multicurves shown in Figure \ref{fig element4}, and it follows again from Theorem \ref{th:skein1} that 
\[x_\zg x_{\zg'}=\pm y(\zeta:C) x_\zeta\pm y(\za:C) x_\za \pm y(\zb:C) x_\zb \pm y(\{\zs,\rho\}:C) x_\zs x_\rho . 
\]
Again, solving for $x_\zeta$  shows that $x_\zeta\in \mathcal{A}$.
\begin{figure}
\scalebox{0.7}{\input{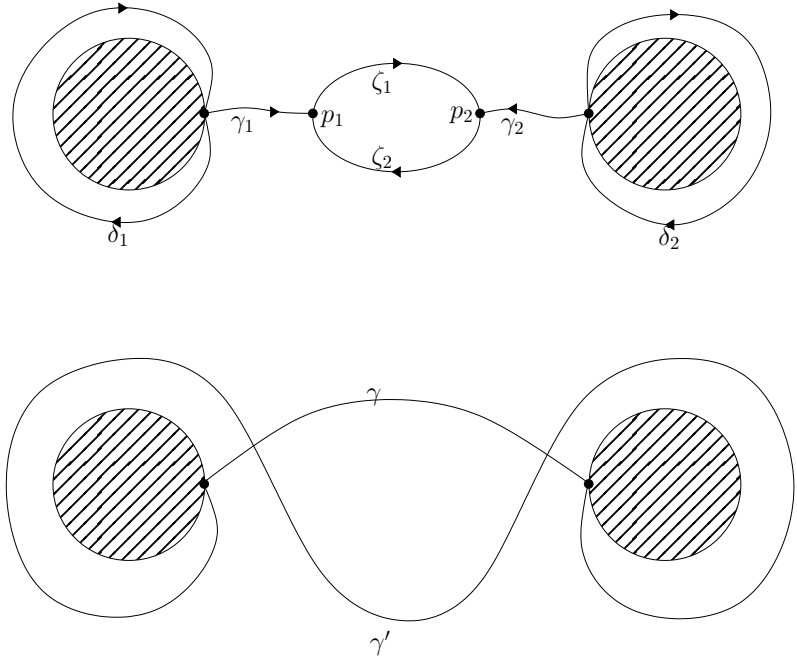}}
\caption {Two arcs $\zg,\zg'$ associated to the essential loop $\zeta$. The smoothing of the multicurve $\{\zg,\zg'\}$ is shown in Figure \ref{fig element4}.}\label{fig element3}
\end{figure}

\begin{figure}
\scalebox{0.7}{\input{figelement3.pstex_t}}
\caption {Smoothing of the multicurve $\{\zg,\zg'\}$ of Figure \ref{fig element3}.}\label{fig element4}
\end{figure}

This shows that for every essential loop $\zeta$ the element $x_\zeta$ is in the cluster algebra.
The element $x_{\Bang_k\zeta} $ is a power of $x_\zeta$, which shows that it also lies in the cluster algebra. This shows that $\B^\circ\subset \A.$  Now Proposition \ref{prop Cheby} implies that $\B\subset \A$. 
\end{proof}

\begin{cor}\label{cor:containment}
If the surface  has genus zero, then $\B^\circ$ and $\B$ are subsets of $\A$.\qed
\end{cor}

\subsection{$\B^\circ$ and $\B$ are spanning sets for $\A$}\label{sect span}
\begin{lemma}
The sets $\B^\circ$ and $\B$ are both spanning sets for the cluster algebra $\A$.
\end{lemma}
\begin{proof}
We start by showing the result for $\B^\circ$. Since the elements of the cluster algebra are polynomials in the cluster variables, it suffices to show that any finite product of cluster variables can be written as a linear combination of elements of $\B^\circ$. 

We will prove the more general statement that for any multicurve $C$, the element $x_C=\prod_{c\in C}x_c$ can  be written as a linear combination of elements of $\B^\circ$. 
If there are no   crossings between the elements of $C$, then $x_C\in\B^\circ$, and we are done.
Suppose therefore that there are exactly $d$   crossings between the elements of $C$. Using Theorem \ref{th:skein1}, we can write
\[x_C=\pm Y_+ x_{C_+} \pm Y_-x_{C_-}\]
where $Y_+$ and $Y_{-}$ are coefficient monomials, while $C_+$ and $C_-$ are multicurves each of which has at most $d-1$ crossings between its elements. The statement for $\B^\circ$ now follows by induction. 

To show the statement for $\B$, we use Propositions \ref{prop cheb} and \ref{prop Cheby}, which show that, for each bangle $\Bang_k\zeta$, we can write $x_{\Bang_k \zeta}$ as a positive integer linear combination of elements of $\B$. Since $\B^\circ$ is a spanning set, it follows that $\B$ is too.
\end{proof}

\begin{remark}
While $\B$ is expected to be an atomic basis, $\B^\circ$ is definitely not atomic. In particular, $x_{\Brac_k \zeta}$ is in $\A^+$ (it expands positively in terms of every cluster), but its expansion in the basis $\B^\circ$ uses the polynomial $T_k(x)$, which has negative coefficients.
\end{remark}

By comparing our construction of
the basis $\mathcal{B}$ with that of Fock and Goncharov, we obtain the
following result.

\begin{corollary}\label{cor 4.9} For a coefficient-free cluster algebra $\A$ from an unpunctured surface
with at least two marked points, the upper cluster algebra and the
cluster algebra coincide. Moreover, the sets $\B$ and $\B^\circ$ are both bases of $\A$.
\end{corollary}

\begin{proof}
It follows from \cite[Theorem 4.11, Proposition 4.12]{MW} that the set $\B$ coincides with the basis of the upper cluster algebra constructed in \cite{FG1}. 
Proposition \ref {lem:containment} ensures that
$\mathcal{B}$ is a subset of the cluster algebra rather than
simply the upper cluster algebra.  Therefore $\B$ is a basis for the cluster algebra and for the upper cluster algebra, and the two algebras coincide.
\end{proof}

\subsection{$\B^\circ$ and $\B$ are linearly independent sets}
It remains to show the linear independence of the sets $\B^\circ$ and $\B$. This is done in Sections \ref{sec:lattice} and \ref{sec:g}.

\section{Lattice structure of the 
matchings of  snake and band graphs}\label{sec:lattice}

In this section we describe the structure of the 
set of  perfect matchings of a snake graph, and the set of good matchings of 
a band graph. 
The main application of our
analysis of matchings is the proof of Theorem 
\ref{g-loops} below.  In Section \ref{sec:g}, we will use this  theorem 
to extend the definition of 
$\gg$-vector to all elements of $\B$ and $\B^{\circ}$.

\begin{theorem} \label{g-loops}
Any element $z$ 
of $\B^{\circ}$ or $\B$
 contains a unique term $\x^g$  not divisible
by any coefficient variable, and the exponent vector of 
each other term is obtained from $g$
by adding a non-negative linear combination of columns 
of $\widetilde{B_T}$.
The same is true if we replace $z$ by any product of 
elements in $\B^{\circ}$ or $\B$.
\end{theorem}

Let $G$ be a snake or band graph with tiles $G_1,\dots,G_n$.  Let $P_-$ denote the 
minimal matching of $G$.  Given an arbitrary matching $P$ of $G$,
its {\it height function} or {\it height monomial}
is the monomial $\prod_{G_i} w_i$ where $G_i$ ranges over all tiles
enclosed by $P \cup P_-$.
We define a {\it twist} 
of a matching $P$ 
to be a local move 
affecting precisely one tile $T$ of $G$, replacing the two 
horizontal edges of $T$ with the two vertical edges, or vice-versa.

The following theorem is a consequence of \cite[Theorem 2]{ProppLattice}.
See Figure \ref{fig:lattice}.

\begin{theorem}\label{lattice}
Consider the set of all perfect matchings of a snake graph $G$
with tiles $G_1,\dots,G_n$.  Construct a graph $L(G)$ whose vertices
are labeled by these matchings, and whose edges connect two vertices
if and only if the two matchings are related by a twist.
This graph is the Hasse diagram of a distributive lattice, whose
minimal element is $P_-$.
The lattice is graded by the degree of each height monomial.
\end{theorem}

\begin{figure}[t]
\begin{centering}
\includegraphics[height=3in]{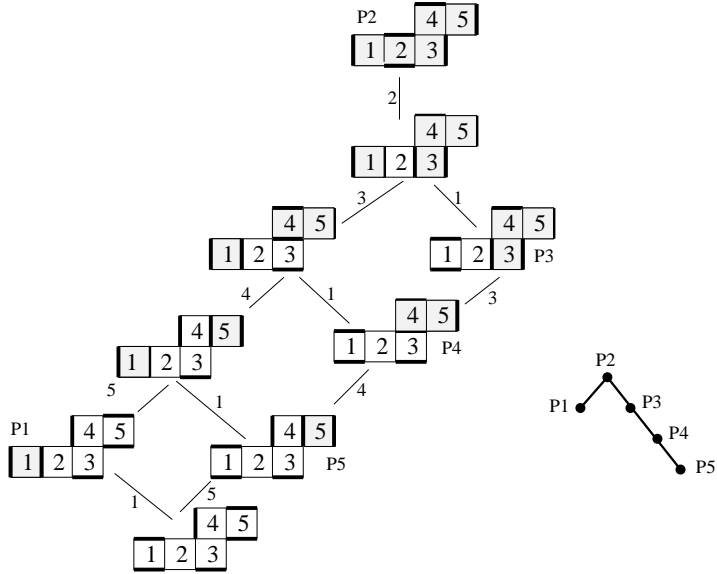}
\par\end{centering}
\caption{ 
Lattice of perfect matchings of a snake graph}
\label{fig:lattice}
\end{figure}

We now prove some more properties of $L(G)$.
We describe how to read off from $G$ a  poset $Q_G$
whose lattice of order ideals $J(Q_G)$ is equal to $L(G)$.  

Given a snake graph $G$, we define a {\it straight} subgraph of $G$
to be a subgraph $H$ formed by consecutive tiles which all lie
in a row or in a column.  We define a {\it zigzag} subgraph $H$ of $G$
to be a subgraph formed by consecutive tiles such that no three
consecutive tiles in $H$ lie in a row or in a column.

\begin{definition}\label{Q_G}
Let $G$ be a snake graph, with tiles $G_1,\dots,G_n$
(labeled from southwest to northeast).  Group the tiles of $G$
into overlapping connected subsets of tiles $S_1,\dots,S_k$,
where each $S_i$ is either a maximal-by-inclusion 
straight or  zigzag subgraph, 
and the $S_i$'s alternate between straight and zigzag subgraphs.
We  associate to $G$ (the Hasse diagram of) a poset $Q=Q_G$ as follows (see Figure \ref{fig:lattice}):
the elements of the poset are labeled $P_1,\dots,P_n$, and there 
is an edge in the Hasse diagram of $Q$ between $i$ and $i+1$.  Suppose $S_i$ 
consists of tiles $G_{r},G_{r+1},\dots,G_{s}$.
If $S_i$ is a zigzag subgraph, 
then the edges of the Hasse diagram between $r$ and $r+1$, $r+1$ and $r+2,...$, $s-1$ and $s$,
are all either oriented
northeast or all oriented southeast.  And if $S_i$ is a straight
subgraph, then the edges of the Hasse diagram between $i_1,\dots,i_r$ alternate between
northeast and southeast orientations.  
If the tile $G_2$ is to the right of (respectively, above) the  tile $G_1$,
the edge from $1$ and $2$
is oriented northeast (respectively, southeast).
\end{definition}
Note that the snake graph in 
Figure \ref{fig:lattice}
 consists of a straight subgraph
$S_1$ consisting of tiles $G_1,\dots,G_3$, and a zigzag subgraph
$S_2$ consisting of tiles $G_2,\dots,G_5$.

\begin{theorem}\label{twisttheorem}
Let $G$ be a snake graph, with tiles  $G_1,\dots,G_n$.  
We assume that the tile $G_1$ is chosen to have {\it positive relative orientation}
(see Definition \ref{relorient}).
Then $L(G)$ is the lattice of order ideals 
$J(Q_G)$ of the poset $Q_G$ from  Definition \ref{Q_G};
the support of the height monomial of a matching in $L(G)$ is precisely 
the elements in the corresponding order ideal.
Moreover, the \emph{twist-parity condition} is satisfied:
if $i$ is odd (respectively, even),
a twist on tile $G_i$ going up in the poset replaces 
the horizontal edges in $G_i$ with the vertical edges
(respectively, the vertical edges with the horizontal edges).
\end{theorem}

\begin{proof}
We use induction on the number of tiles.
If $G$ is composed of tiles $G_1,\dots,G_n$,
there are two cases: either $G_n$ is 
to the right of $G_{n-1}$, or is 
directly above tile $G_{n-1}$.
We consider the first case (the second case is similar, so we omit it).
Let $H_1$ be the subgraph of $G$ consisting of tiles $G_1,\dots,G_{n-1}$.
Note that each perfect matching of $H_1$ can be extended uniquely
to a perfect matching of $G$ by adding the rightmost vertical 
edge of $G_n$. We call these {\it Type 1} matchings of $G$.
Now, consider perfect matchings of $G$ which use
the two horizontal edges of $G_n$: we call these {\it Type 2} matchings.
Recall the decomposition of $G$ as a union of subgraphs 
$S_1,\dots,S_k$ from Definition \ref{Q_G}.  Suppose that $S_k$
consists of tiles $G_r, G_{r+1},\dots,G_n$.
If $S_k$ is a zigzag subgraph, then Type 2 perfect matchings will be forced
to include every other edge of the boundary of $G_{r+1} \cup \dots \cup G_n$, 
and indeed,
will be in bijection with perfect matchings of the subgraph $H_2$ of $G$
consisting of tiles $G_1,\dots,G_{r-1}$.  If $S_k$ is a straight subgraph, then Type 2 perfect matchings
will be in bijection with perfect matchings of the subgraph $H_2$ of $G$
composed of tiles $G_1,\dots,G_{n-2}$.

In Figure \ref{fig:lattice}, there are two Type 2 perfect matchings,
$P1$ and the minimal element in the poset. These perfect matchings
are in bijection with matchings of $H_2$, which in this case consists
of just tile $G_1$.  The other perfect matchings are of Type 1. 

The set of Type 1 matchings forms a sublattice $L_1$ of $L(G)$ (isomorphic
to $L(H_1)$), and the set of Type 2 matchings forms a sublattice $L_2$ of $L(G)$
(isomorphic to $L(H_2)$).  By induction, within $L_1$ and $L_2$, 
the twist-parity condition 
is satisfied (note that within $L_1$ and $L_2$ there
are no twists involving tile $G_n$).
The lattice $L(G)$ is equal to the disjoint union
of $L_1$ and $L_2$  together with some edges connecting them,
which correspond to twists on tile $G_n$.  
If $n$ is odd (respectively, even), then the minimal matching $P_-$ of $G$ uses
one or both of the horizontal (respectively, vertical) 
edges of $G_n$.  Therefore when $n$ is odd (respectively, even), 
if $P$ is a matching of $G$ which uses both horizontal (respectively, vertical)
edges of $G_n$, performing a twist will increase the height function.
This proves the twist-parity condition.

To prove that $L(G)\cong J(Q_G)$,
we use the decomposition
$G = S_1 \cup \dots \cup S_k$.  First suppose that 
$S_k$ is a straight subgraph.  
If $n$ is even then the Type 1 matchings do not contain
$w_n$ in their height monomial, and by induction they are in bijection
with order ideals in $Q_{H_1}$, that is, order ideals of $Q_{G}$ 
which do not use $n$.  The Type 2 matchings {\it do} contain
$w_n$ and also $w_{n-1}$ in their height monomial, because $S_k$ is straight and $k$ is even.  By induction
they are in bijection with order ideals in $Q_{H_2}$, which in turn
are in bijection with order ideals of $Q_G$ which involve $n$ and $n-1$.
Together, this gives
a decomposition of the order ideals of $Q_G$ as a disjoint union
of the Type 1 and Type 2 matchings, which proves that $L(G) \cong
J(Q_G)$.  When $n$ is odd the argument is 
similar, but this time it is the Type 1 matchings whose height monomial
contains $w_n$.

Now suppose that $S_k$ is a zigzag subgraph.
Write $S_k = G_r \cup G_{r+1} \cup \dots \cup G_n$.
If $n$ is even then the Type 1 matchings do not contain $w_n$
in their height monomial, and by induction they are in bijection
with order ideals in $Q_{H_1}$, which in turn are in bijection
with order ideals of $Q_G$  which do not use $n$.
The Type 2 matchings must contain $w_r, w_{r+1},\dots,w_n$
in their height monomials, and by induction are in bijection
with order ideals in $Q_{H_2}$, which in turn are in bijection
with order ideals of $Q_G$ which involve $n$ (and hence $n-1,n-2,\dots,r$.)
Together, this gives
a decomposition of the order ideals of $Q_G$ as a disjoint union
of the Type 1 and Type 2 matchings, which proves that $L(G)$
is isomorphic to $J(Q_G)$.  When $n$ is odd the argument is similar,
but this time the height monomials of the Type 1 matchings contain $w_n$,
and the height monomials of the Type 2 matchings do not contain 
any of $w_r,w_{r+1},\dots,w_n$.
\end{proof}

\begin{remark}
If $\mathcal{Q}_T$ is the quiver of the triangulation $T$, then each generalized arc $\zg$ defines a string module $M(\zg)$ over the corresponding Jacobian algebra, see \cite{BZ}. The string of $M(\zg)$ is precisely the poset $Q$ and the lattice $L(G)$ is the lattice of string submodules of $M(\zg)$.
\end{remark}

We now consider the good matchings of a band graph $\widetilde{G}$,
where $\widetilde{G}$ is obtained from a snake graph $G$ by identifying
two edges.
By Remark \ref{descend}, we can identify the good matchings of 
$\widetilde{G}$ with a subset of the perfect matchings of $G$,
so in particular, we can consider the subgraph $L(\widetilde{G})$
of $L(G)$ which is obtained from $L(G)$ by restricting to the good
matchings.  As we now explain, $L(\widetilde{G})$ has the structure
of a distributive lattice, that is, we can identify it with 
the lattice of order ideals of a certain poset.

\begin{definition}\label{poset-band}
Let $\widetilde{G}$ be a band graph obtained from a snake graph 
$G$ with tiles $G_1,\dots,G_n$.  There are four different cases,
based on the geometry of how $x$ and $y$ sit in the first and last tile
of $\widetilde{G}$, see Figure \ref{fig band}.  Let $Q=Q_G$ be the 
poset associated to $G$ by Definition \ref{Q_G}.  We now let 
$\widetilde{Q} = \widetilde{Q}_G$ be the poset obtained from 
the poset $Q=Q_G$ by imposing one more relation:
in Cases 1 and 2, we impose the relation $1 > n$;
and in Cases 3 and 4, we impose the relation $1 < n$.
(It is straightforward to verify that $\widetilde{Q}$ is still 
a well-defined poset.)
\end{definition}

We have the following analogue of Theorem \ref{twisttheorem} for 
band graphs.
\begin{theorem}\label{BandTheorem}
Let $\widetilde{G}$ be a band graph obtained from the 
snake graph $G$, with tiles  $G_1,\dots,G_n$.  
We assume that tile $G_1$ is chosen to have {\it positive relative orientation}.
Then $L(\widetilde{G})$ is the lattice of order ideals 
$J(\widetilde{Q}_G)$ of the poset $\widetilde{Q}_G$ from  
Definition \ref{poset-band};
the support of the height monomial of a matching in $L(\widetilde{G})$ is precisely 
the elements in the corresponding order ideal.
Since $L(\widetilde{G})$ is a subgraph of $L(G)$, 
the \emph{twist-parity condition} is satisfied.
\end{theorem}

\begin{proof}
While there are four cases to consider, 
the proofs in all cases are essentially the same, so 
we just give the proof in Case 1 --  the case that 
$G$ and $\widetilde{G}$ are as in the left of Figure \ref{band-case1}
(so in particular $G$ has an odd number of tiles).
\begin{figure}
\input{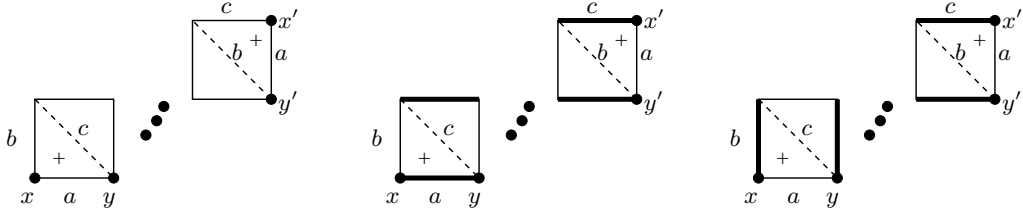}
\caption{Illustrating the proof of Theorem \ref{BandTheorem}}
\label{band-case1}
\end{figure}
Then the minimal matching of $G$ contains the edge
between $x$ and $y$, and does {\it not} use the edge
between $x'$ and $y'$, see the middle picture in Figure \ref{band-case1}.
Every perfect matching of $G$ descends to a good matching of 
$\widetilde{G}$ except those which do not use either the edge
between $x$ and $y$ or the edge between $x'$ and $y'$; see the right picture
in Figure \ref{band-case1}.  Therefore the perfect matchings of $G$ which 
do not descend to good matchings of $\widetilde{G}$ are precisely those
whose height monomial contains $w_1$ but not $w_n$.  Using the identification
of perfect matchings of $G$ with order ideals of $Q_G$, we see that the 
height monomials of good matchings of $\widetilde{G}$ can be identified with 
the order ideals of $Q_G$ which use the element $n$ whenever they use element $1$.
These are precisely the order ideals of $\widetilde{Q}_G$.
\end{proof}

See Figure \ref{fig:bandlattice}
for the lattice of good matchings of a band graph $\widetilde{G}$
obtained from the snake graph $G$ from Figure \ref{fig:lattice}
by identification of the vertices $x$ and $x'$, and $y$ and $y'$.
\begin{figure}[t]
\begin{centering}
\includegraphics[height=3in]{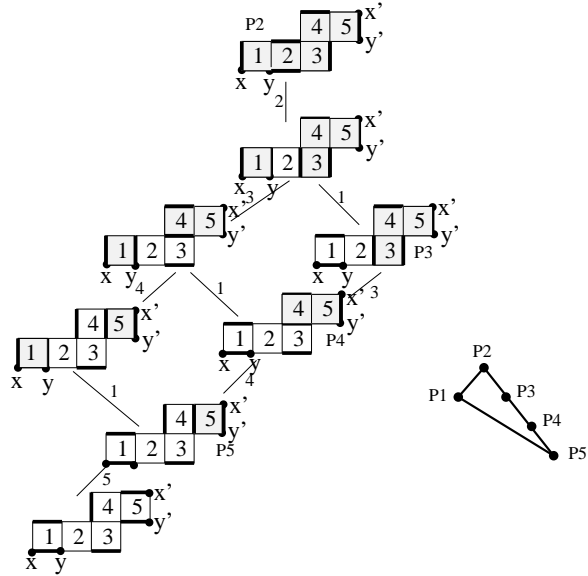}
\par\end{centering}
\caption{Lattice of good matchings of a band graph}
\label{fig:bandlattice}
\end{figure}

\begin{remark}
If $\mathcal{Q}_T$ is the quiver of the triangulation $T$, then each essential loop $\zeta$ defines a family of band modules $M_{\zl,k}(\zeta)$, $\zl\in\mathbb{P}^1,k\ge1$, over the corresponding Jacobian algebra, see \cite{BZ}. The band  is precisely the poset $Q$ and the lattice $L(G)$ is the lattice of string submodules of $M_{\zl,1}(\zeta)$ together with $M_{\zl,1}(\zeta)$.

The bangle $\Bang_k(\zeta)$ corresponds to the direct sum of $k$ copies of $M_{\zl,1}(\zeta)$. If the surface is a disk or an annulus, then the basis $\B^\circ$ corresponds to the generic basis in \cite{Dgeneric,GLS}.

On the other hand, the bracelet $\Brac_k(\zeta)$ does not have a module interpretation; it does \emph{not} correspond to the band module $M_{\zl,k}(\zeta)$.
\end{remark}

Finally we turn to the proof of Theorem \ref{g-loops}.
\begin{proof}
Let $\widetilde{B} = \widetilde{B_T}$ be the extended exchange matrix.
Note that if any two cluster algebra elements $z_1$ and $z_2$
satisfy the conditions of Theorem \ref{g-loops}, then so does
$z_1 z_2$.  Therefore it suffices to prove Theorem \ref{g-loops} for
cluster variables, and the cluster algebra elements associated to 
essential loops and bracelets.  Theorem \ref{g-loops} for cluster variables follows
from Proposition \ref{g} and the fact that the $F$-polynomials
of cluster variables from surfaces have constant term $1$
(see \cite[Section 13.1]{MSW}).  

By Definition \ref{def closed loop}, 
each cluster algebra element associated to a closed loop  is a
generating function for the good matchings of a band graph.
By Theorem 
\ref{BandTheorem},
there is a sequence of twists from the minimal matching $P_-$ to any 
other good matching $P$ of a band graph, where every twist is a cover 
relation going up in the poset.
Moreover, the twist-parity condition holds: 
along this path, each twist on a tile of positive (respectively, negative)
relative orientation 
will replace horizontal 
edges by vertical edges (respectively, vertical edges by horizontal ones).  
Finally, suppose that $P_2$ is a good matching obtained from $P_1$
by such a twist on tile $G_i$.  Then it follows from our construction
of band graphs 
that the exponent vector of $x(P_2) y(P_2)$ is equal to 
the exponent vector of $x(P_1) y(P_1)$ 
plus the $i$th column of $\widetilde{B}$. 

Note that similar arguments, together with Theorem \ref{twisttheorem},
give a new proof of Theorem \ref{g-loops} for cluster variables
associated to arcs.
\end{proof}
%%%%%%%%%%%%%%%%%%%%%%%%%%%%%%%%
%%
%% SECTION
%%
%%%%%%%%%%%%%%%%%%%%%%%%%%%%%%%%%%
\section{The $\gg$-vector map and linear independence of $\B^\circ$ and $\B$}\label{sec:g}

By Theorem \ref{g-loops} and Remark \ref{rem g}, 
each element of $\B$ and $\B^0$ is homogeneous with respect to 
the $\gg$-vector grading.  The same is true for any product
of elements from $\B$ and $\B^0$.  This
allows us to 
extend the definition of $\gg$-vector to 
all elements of $\B$ and $\B^0$ (and to all products of such elements).

\begin{definition}\label{g-def}
The \emph{$\gg$-vector} of any 
element $x_C$ of $\B$ or $\B^0$,
with respect to the seed
$T$, is the multidegree 
of $x_C$,
using the $\gg$-vector grading.  
Additionally, for every collection $x_j, j\in J$ of elements 
 of $\B$ (or $\B^0$, respectively),
we define $\gg(\prod_j x_j) = \sum_j \gg(x_j)$.
\end{definition}

In Theorem \ref{g-loops}, we have shown that every element of $\B^\circ$ and $\B$ has a unique leading term. For arcs and essential loops, this leading term is given by the minimal matching $P_-$ of the corresponding snake graph. Therefore, we can compute its $\gg$-vector as follows.

\begin{prop}\label{leadingterm}
Let $\gamma$ be an arc or an essential loop.
Then
$x_{\gamma}$ 
has a unique Laurent monomial $\frac{x(P_{-})}{\cross(T, \gamma)}$ which is not divisible by 
any coefficient variable $y_{\tau_i}$.
Moreover, $$\gg(x_{\gamma}) = 
\deg\left(\frac{x(P_{-})}{\cross(T, \gamma)}\right),$$
where $P_{-}$ is the minimal matching of the snake or band graph
associated to $\gamma$ and $T$, 
and $\cross(T,\gamma)$ is the corresponding
crossing monomial. \qed
\end{prop}

\begin{lemma}\label{skein-g}
Let $c_1$ and $c_2$ be  arcs or essential loops, and 
consider the skein relation in $\A$ which writes
$x_{c_1} x_{c_2} = \sum_i Y_i M_i$, where the $M_i$'s are elements
of $\B^{\circ}$
 and the $Y_i$'s are monomials
in coefficient variables $y_{\tau_j}$. 
Then there is a unique $j$ such that $Y_j=1$.  As a consequence,
for each $i \neq j$, the exponent vector of 
$M_i$ is obtained from the exponent vector of $M_j$ by 
adding a non-negative linear combination of columns of 
$\widetilde{B_T}$.
We call the element $M_j$ the \emph{leading term}
in the skein relation
$x_{c_1} x_{c_2} = \sum_i Y_i M_i$.
\end{lemma}

\begin{proof}
The key to the proof is the observation that every skein relation which expresses
a product of crossing arcs or loops 
in terms of arcs and loops which do not cross
has a unique term on the right-hand-side with no coefficient 
variables. Once we have proved this observation, the existence and uniqueness of $j$ follows.
The relationship between $\gg(M_i)$ and $\gg(M_j)$ is then 
a consequence of the fact that  elements of 
$\B^{\circ}$ are homogeneous with respect to the 
$\gg$-vector grading (see Theorem \ref{g-loops}),
which implies that every term in the equation
$x_{c_1} x_{c_2} = \sum_i Y_i M_i$
must have the same $\gg$-vector.

It remains to show the observation above.  Theorem \ref{th:skein1} implies that the skein relations   have the form
\[H_1 = \pm Y_2 H_2 \pm Y_3 H_3,\]
where $Y_2$ and $Y_3$ are monomials in the coefficient variables and
each $H_i$ represents the product of one or two cluster algebra elements,
where those elements are given by our snake and band graph formulas.
In particular, each $H_i$ is in $\Z[x_i^{\pm 1}, y_i]$,
has all coefficients positive, and has a unique term that is not divisible by any of the $y_i$.

It follows from \cite[Lemma 7]{dylan} and Theorem \ref{th:skein1}, that at least one of $Y_2$ and $Y_3$ is equal to 1.
For the sake of contradiction, suppose that both of them are equal to 1.
In that case we have
\[H_1 = \pm H_2 \pm H_3.\]
It is impossible that we have two negative signs on the right hand side,
but we may have one negative sign.
So either $H_1 = H_2 + H_3$ or $H_1 + H_2 =H_3.$ 
Both cases are equivalent after permuting indices, so let
 us suppose without loss of generality that $H_1 + H_2 = H_3.$
Then if we set all the cluster variables equal to 1, and all the coefficient variables equal to 0,
 we get 1 = 1+1.  That is a contradiction.
\end{proof}

\begin{prop}\label{same-g}
Let $\gamma$ be an essential loop in $(S,M)$.  Then
$\Brac_k (\gamma)$ and $\Bang_k (\gamma)$ have the same
$\gg$-vector.
\end{prop}

\begin{proof}
On one hand, we have
\[\gg(\Bang_k \zg) = \gg(x_\zg^k) =k \gg(x_\zg).\]
On the other hand, $\gg(\Brac_k(\zg)) {=}\gg(T_k(x_\zg))$, by Proposition {\ref{prop Cheby}}, and  the result follows from Proposition \ref{prop cheb}.

\end{proof}

Let $e_i$ denote the element of $\Z^n$ with a $1$ in the 
$i$th place and $0$'s elsewhere.
Let $(\tau_1,\dots,\tau_n)$ denote the elements of the initial
  triangulation $T$.  By definition of $\gg$-vectors,
$\gg(x_{\tau_i}) = e_i$ for all $i$.
We now construct
an  element of $\A$ whose
$\gg$-vector is $-e_i$, for each $1 \leq i \leq n$.  

\begin{figure}
\scalebox{0.8}{\input{figlemwo2.pstex_t}}
\caption{The arc $\tbar_i$}\label{figinvert}
\end{figure}

\begin{prop}\label{invert}
Let $i$ be an integer between $1$ and $n$.  Then there exists
an arc $\tbar_i$ of $(S,M)$, such that 
$\gg(x_{\tbar_i})=-e_i$.  The 
arc $\tbar_i$ is constructed
as follows:
Suppose that $\tau_i$ is an arc between two marked points $x$ and $y$,
and let $d_1$ and $d_2$ denote the boundary 
segments such that  $d_1$ is incident to $x$ and is in the 
clockwise direction from $\tau_i$, and $d_2$ is incident to $y$
and is in the clockwise direction from $\tau_i$.  Let $x'$ and $y'$
be the other endpoints of $d_1$ and $d_2$, besides $x$ and $y$.
Let $\tbar_i$ be the arc of $(S,M)$ between points $x'$ and 
$y'$, which is homotopic to the concatenation of 
$d_2, \tau_i, d_1$.    See Figure \ref{figinvert}.
\end{prop}
\begin{proof}
  Let $r$ and $s$ be the arcs
in $(S,M)$ from $x$ to $y'$ and $x'$ to $y$, respectively, obtained
by resolving the crossing between $\tbar_i$ and $\tau_i$.
Then we have the exchange relation
$x_{\tau_i} x_{\tbar_i} = Y x_r x_s + 1$, where $Y$
is a monomial in $y_{\tau_j}$'s.  Note that the term $1$ comes
from  the two boundary segments obtained by resolving the crossing between $\tbar_i$ and $\tau_i$
in the other direction.
Since cluster variables are homogeneous elements with respect
to the $\gg$-vector grading, it follows that 
$\gg(x_{\tau_i} x_{\tbar_i}) = 0$.
It follows that 
$\gg(x_{\tbar_i}) = -\gg(x_{\tau_i}) = -e_i$, as desired.
\end{proof}
\begin{remark}
In the corresponding cluster category, the arc $\tbar_i$ corresponds to to the Auslander-Reiten translate of the arc $\tau_i$, see \cite{BZ}.
\end{remark}

\subsection{Fans}
Let $T$ be a triangulation and $\zg$ be an arc or a closed loop.
Let $\zD$ be a triangle in $T$ with sides $\zb_1,\zb_2$, and  $\tau$, that is crossed by $\zg$ in the following way: $\zg$ crosses  $\zb_1 $ at the point $p_1$ and crosses $\zb_2$ at the point $p_2$, and the segment of  $\zg$ from $p_1$ to $p_2$ lies entirely in $\zD$, see the left of Figure \ref{figv}. Then there exists a unique vertex $v$ of the triangle $\zD$ and a unique contractible closed curve $\ze$ given as the homotopy class of a curve starting at the point $v$, then following $\zb_1$
until the point $p_1$, then following  $\gamma$ until the point $p_2$ and then following $\zb_2$ until $v$. We will use the following notation to describe this definition: 

\[\epsilon =\xymatrix{ v \ar@{-}[r]^{\zb_1} &p_1\ar@{-}[r]^\zg & p_2 \ar@{-}[r]^{\zb_2} &v}.\]

\begin{definition} \label{def fan}
A \emph{$(T,\gamma)$-fan with vertex $v$} a collection of arcs $\beta_0,\beta_1,\ldots,\beta_k$, with $\beta_i\in T$ and $k\ge 0$ with the following properties (see the right of Figure \ref{figv}):
\begin{enumerate}
\item $\gamma$ crosses $\beta_0,\beta_1,\ldots,\beta_k$ in order at the points $p_0,p_1,\ldots, p_k$, such that $p_i$ is a crossing point of $\gamma$ and $\beta_i$, and the segment of $\zg$ from $p_0$ to $p_k$ does not have any other crossing points with $T$;
\item each $\beta_i$ is incident to $v$;
\item for each $i<k$, let $\ze_i$  be the unique contractible closed curve given by 
\[\xymatrix{ v \ar@{-}[r]^{\zb_i} &p_i\ar@{-}[r]^-\zg & p_{i+1} \ar@{-}[r]^-{\zb_{i+1}} &v};\]
then 
 for each $i<k-1$, the concatenation of the curves $\ze_i\ze_{i+1}$ is homotopic to 
\[\xymatrix{ v \ar@{-}[r]^{\zb_i} &p_i\ar@{-}[r]^-\zg&p_{i+1}\ar@{-}[r]^-\zg & p_{i+2} \ar@{-}[r]^-{\zb_{i+2}} &v}.\]
\end{enumerate}
\end{definition}
Condition (3) in the above definition is equivalent
to the condition that \[\xymatrix{ v \ar@{-}[r]^{\zb_i} &p_i\ar@{-}[r]^-\zg & p_{i+2} \ar@{-}[r]^-{\zb_{i+2}} &v}\]
is contractible.
\begin{definition} A $(T,\gamma)$-fan
 $\beta_0,\beta_1,\ldots,\beta_k$ 
 is called \emph{maximal} if there is no arc $\za\in T$ such that 
 $\beta_0,\beta_1,\ldots,\beta_{k},\za$ or  $\za,\beta_0,\beta_1,\ldots,\beta_{k}$ is a  $(T,\gamma)$-fan.
 \end{definition}
 
 Every $(T,\zg)$-fan  $\zb_0,\zb_1,\ldots,\zb_k$ defines a triangle with simply connected interior whose vertices are $v,p_0,p_k$ and whose boundary is the contractible curve 
\[\xymatrix{ v \ar@{-}[r]^{\zb_0} &p_0\ar@{-}[r]^\zg & p_k \ar@{-}[r]^{\zb_k} &v}.\]
The orientation of the surface $S$ induces an orientation on this triangle, and we say that $\zb_0$ is the \emph{initial} arc  and $\zb_k$ is the \emph{terminal} arc of the fan, if going around the boundary of the triangle along the curve $\xymatrix{ v \ar@{-}[r]^{\zb_0} &p_0\ar@{-}[r]^\zg & p_k \ar@{-}[r]^{\zb_k} &v}$
is clockwise. In the fan $\tau_1,\tau_2,\tau_3,\tau_2$ in the example given on the right of Figure \ref{figfan}, the initial arc is $\tau_2$ and the terminal arc is $\tau_1$.

\begin{figure}\begin{center}
\input{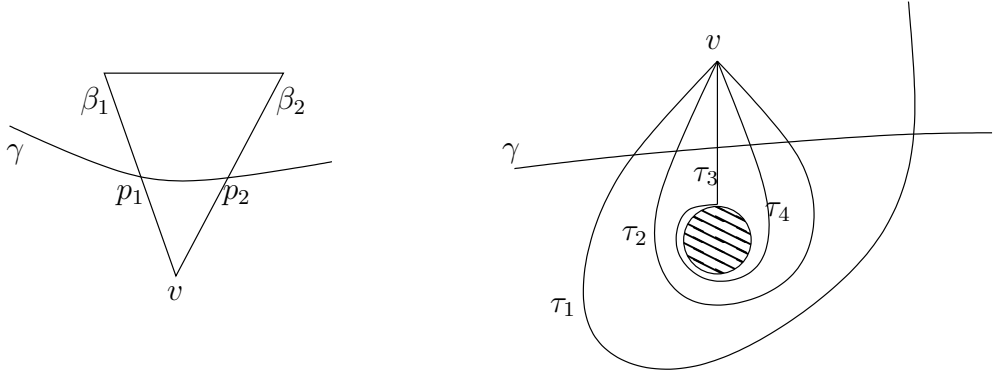}
\caption{Construction of $(T,\zg)$-fans (left). The fan $\tau_1,\tau_2,\tau_3,\tau_4,\tau_2$ (right) can not be extended to the right, because the configuration $\tau_1,\tau_2,\tau_3,\tau_4,\tau_2,\tau_1$ does not satisfy condition (3) of Definition \ref{def fan}}\label{figfan}\label{figv}\end{center}
\end{figure}

\subsection{Multicurves and leading terms}
Recall from Section \ref{sect span} that given any multicurve $\{\zg_1,\ldots,\zg_t\}$, 
we can always apply a series of 
smoothings to replace it with 
a union of simple multicurves, called the {\it smooth resolution} of  
$\{\zg_1,\ldots,\zg_t\}$.  
In the cluster algebra, taking the resolution of the multicurve $\{\zg_1,\zg_2,\ldots,\zg_t\}$ 
corresponds to applying skein relations to the product $x_{\zg_1}x_{\zg_2}\cdots x_{\zg_t}$ until the result is a linear combination of elements of $\B^\circ$. 
Also recall that by Lemma \ref{skein-g}, if we write the product $x_{\zg_1}x_{\zg_2}\cdots x_{\zg_t}$ as a linear combination of $\B^\circ$, then there is a unique term with trivial coefficient, 
say $x_{\za_1}x_{\za_2}\cdots x_{\za_s}$, which is called the \emph{leading term}.
We say that the multicurve $\{\za_1,\za_2,\ldots,\za_s\}$ is \emph{equivalent to the leading term of the resolution of}  $\{\zg_1,\zg_2,\ldots,\zg_t\}$. 
Note that any boundary segment $b$, which appears during the process, is not included in the multicurves, since the corresponding element $x_b$ in the cluster algebra is equal to $1$.

\subsection{An inverse for the $\gg$-vector map} \label{sec:g-inverse}
In this subsection, we use the $(T,\zg)$-fans to prove that the $g$-vector map is a bijection between $ \B^{\circ}$ and $\ZZ^n$. We will define a map $f:\Z^n \to \B^{\circ}$ and show that it is the inverse of the $\gg$-vector map. Recall that for an arc $\tau_i$, we denote by $\tbar_i$ the unique arc    whose $g$-vector is  $-e_i$. 

\begin{definition}
Let $v=(v_1,\dots,v_n)\in \Z^n$, and write it uniquely
as $v=\sum_i r_i e_i + \sum_j s_j (-e_j)$,
where $i$ ranges over all coordinates of $v$ with $v_i >0$,
and $j$ ranges over all coordinates of $v$ with $v_i < 0$.
So $r_i = v_i >0$, and $s_j = -v_j >0$.
Then use the skein relations to write
$\prod_i (x_{\tau_i})^{r_i} \prod_j (x_{\taubar_i})^{s_i}$
as a linear combination of elements in $\B^{\circ}$.
Define $f(v)$ to be the leading monomial in this sum,
as defined by Lemma \ref{skein-g}.
\end{definition}

\begin{lemma}\label{lem gf}
The composition $\gg \circ f$ is the identity map from 
$\Z^n$ to itself, and so $\gg$ is surjective and $f$ is injective.
\end{lemma}
\begin{proof}
For $v \in \mathbb{Z}^n$, we have $\gg(f(v)) = 
\gg(\prod_i (x_{\tau_i})^{r_i} \prod_j (x_{\tbar_i})^{s_i})$,
thus, by Definition \ref{g-def}, $\gg(f(v)) = v$.
\end{proof}

\begin{lemma}\label{lem g1}
Let $\zg$ be an arc. Choose an orientation of $\zg$, and let $s$ be its starting point and $t$ its ending point. Denote by $\zd_s$ the arc that is clockwise from $s$ in the first triangle of $T$ that $\zg $ meets, and denote by $\zd_t$  the arc that is clockwise from $t$ in the last triangle that $\zg $ meets. 
Let $F_1,\ldots,F_\ell$ be the maximal $(T,\zg)$-fans ordered by the orientation of $\zg$ and let $\zs_i$ be the initial arc of $F_i$ and $\tau_i$ the terminal arc of $F_i$.
  \begin{enumerate}
\item If $\zg$ crosses the initial arc of $F_1$ first then $\zg$ is equivalent to the leading term in the resolution of the multicurve
\[ \{\zd_s,\zd_t, \sbar_i, \tau_i , \sbar_\ell \mid \textup{$i$ is an odd integer with $1\le i<\ell$}\}.
\]

\item If $\zg$ crosses the terminal arc of $F_1$ first then $\zg$ is equivalent to the leading term in the resolution of the multicurve
\[ \{\zd_s,\zd_t, \sbar_i, \tau_i , \sbar_\ell \mid \textup{$i$ is an even integer with $2\le i<\ell$}\}.
\]
\end{enumerate}

\end{lemma}

\begin{proof}
We may assume without loss of generality that $\zg$ crosses the initial arc of $F_1$ first. 
Note first that $\zs_i=\zs_{i+1} $  for all even $i<\ell$, and $\tau_i=\tau_{i+1}$ for all odd $i<\ell$.
We proceed by induction on $\ell$. Suppose first that $\ell=1$. Then $\{\zd_s,\zd_t,\sbar_1\}$ is the multicurve shown on the left of Figure \ref{fig lemg1}, where boundary segments are labeled $b$.

The leading term of the resolution of this multicurve is shown on the right   of Figure \ref{fig lemg1}, and we see that it  is equivalent to $\zg$.

 \begin{figure}
\scalebox{0.8}{\input{figlemg1-1.pstex_t}}
\caption{Proof of Lemma \ref{lem g1} for $\ell=1$}
\label{fig lemg1}
\end{figure}

Now suppose that $\ell>1$. The smoothing at the first crossing point $p_1$  of $\zg $ and $\zs_1$ has the leading term $\{\zd_s , \zg'\}$, where $\zg'$ is the arc starting at the vertex $s'$ of the first fan $F_1$, following $\zs_1$ up to the   point $p_1$ and then following $\zg$ until the endpoint $t$, see Figure \ref{fig lemg1a}.
\begin{figure}
\scalebox{1}{\input{figlemg2.pstex_t}}
\caption{Proof of Lemma \ref{lem g1} for $\ell>1$}
\label{fig lemg1a}
\end{figure}
 Note that $\zg'$ is avoiding all the crossings with the fan $F_1$.  
Thus the maximal $(T,\zg')$-fans $F_2',F_3',\ldots,F_\ell'$ are  given by $F_i'=F_i$, for $i>2$, and $F_2'$ is obtained from $F_2$ by removing the terminal arc $\tau_2$.
By induction, we know that $\zg'$ is equivalent to the leading term of the resolution of the multicurve 
\[ \{\tau_1=\zd_{s'},\zd_t, \sbar_i, \tau_i , \sbar_\ell \mid \textup{$i$ is an odd integer with $3\le i<\ell$}\}.
\]
On the other hand, the leading term of the resolution of $\{\zd_s, \sbar_1,\zg'\}$ is equivalent to $\zg$, and the result follows.
\end{proof}

\begin{lemma}\label{lem g2}
Let $\zg$ be a closed loop. 
Let $F_1,\ldots,F_\ell$ be the maximal $(T,\zg)$-fans ordered by the orientation of $\zg$ and let $\zs_i$ be the initial arc of $F_i$ and $\tau_i$ the terminal arc of $F_i$.
Then $\zg$ is equivalent to the leading term in the resolution of the multicurve
\[  \{\sbar_i, \tau_i \mid\textup{$i$ is an odd integer with $1\le i\le \ell-1$}\},
\]
which is the same as
\[  \{\sbar_i, \tau_i \mid\textup{$i$ is an even integer with $2\le i\le \ell$}\}.
\]
\end{lemma}

 \begin{proof}
First note that, since $\zg$ is closed loop, the number of maximal fans whose vertex lies in the interior of $\zg$ must be equal to the number of maximal fans whose vertex lies in the exterior of $\zg$; thus $\ell $ is even.
Choose a starting point $p$ and an orientation for $\zg$ such that the first arc that $\zg$ crosses is the terminal arc $\tau_1$ of the fan $F_1$ in the point $x$, and  then $\zg$ crosses the fan $F_1$. Note that $\tau_\ell=\tau_1$, since $\zg $ is a closed loop. 
Smoothing the multicurve $\{\tau_\ell,\zg\}$, we get a leading term $\zg'$ that is an arc starting at a point $s$, following $\tau_\ell$ up to the point $x$, then following $\zg$ one time around up to the point $x$ again and then following $\tau_\ell$ until its endpoint which we label $t$.
 
Lemma \ref{lem g1} implies that $\zg'$ is equivalent to the leading term of the resolution of the multicurve
  \[  \{\zd_s,\zd_t, \sbar_i, \tau_i , \sbar_\ell \mid \textup{$i$ is an even integer with $2\le i<\ell$}\}.
 \]
Note that $\zd_s=\zd_t=\tau_\ell$.
On the other hand, $\zg$ is equivalent to the leading term of the resolution of the multicurve $\{\zg',(-\tau_\ell)\}$, and the result follows since the leading term of $\{(-\tau_\ell),\tau_\ell\}$ is equivalent to a union of boundary segments.
\end{proof}

\begin{theorem}\label{gbijection}
The
 $\gg$-vector maps $\gg: \B^{\circ}\to \Z^n$ 
and $\gg: \B \to \Z^n$ are both bijections.  
\end{theorem}

\begin{proof}
By Proposition \ref{same-g}, it suffices to show that 
$\gg:\B^{\circ}\to \Z^n$ is a bijection.
Lemmas \ref{lem g1} and \ref{lem g2} imply that each arc and each closed loop lies in the image of $f$, which allows us to conclude that $f$ is surjective. We have shown in Lemma \ref{lem gf} that $\gg\circ  f$ is the identity on $\ZZ^n$, which shows that $f$ is a bijection and $\gg=f^{-1}$. 
\end{proof}

\begin{corollary}\label{cor:independence}
The sets $\B^{\circ}$ and $ \B$ are  both
linearly independent.
\end{corollary}

\begin{proof}
Clearly the extended $2n \times n$ exchange matrix
$\widetilde{B_T}$
associated to $T$, whose bottom $n\times n$ submatrix
consists of the identity matrix, has linearly independent columns.
Let $x_1,\dots,x_n$ denote the cluster variables $x_{\tau_1},\dots, x_{\tau_n}$,
and $x_{n+1},\dots,x_{2n}$ denote the coefficient variables
$y_{\tau_1},\dots,y_{\tau_n}$.

Proposition
\ref{leadingterm} implies that 
if $\gamma$ is any arc, essential loop or bracelet, then $x_{\gamma}$ has a unique 
term
$x_M$ which is a Laurent monomial in $x_1,\dots,x_n$ and which is not divisible by any coefficient variable $y_{\tau_i}$.
Proposition \ref{g} and Theorem \ref{g-loops} imply that the exponent vector 
of every other Laurent monomial in the expansion of $x_{\gamma}$
can be obtained from the exponent vector of $x_M$ by adding
a non-negative linear combination of columns of $\widetilde{B_T}$.
This means that $x_M$ is the leading term of each Laurent expansion.
Finally, Theorem 
\ref{gbijection} implies that the exponent vectors of the 
leading terms of all elements of $\B^{\circ}$ are pairwise distinct. 
Proposition \ref{suffice} now implies that elements of 
$\B^{\circ}$ are linearly independent.  The same proof
works for $\B$.
\end{proof}

For completeness, we include the following result on the computation of $\gg$-vectors.
\begin{corollary}
\begin{enumerate}
\item The $g$-vector of an arc is equal to $e_{\zd_s}+e_{\zd_t}-e_{\zs_\ell}+\sum (e_{\tau_i}-e_{\zs_i})$, where  $\zs_i,$ respectively $\tau_i$ is the initial, respectively terminal, arc of the $i$-th fan, and the sum is taken over all maximal $T$-fans $F_i$ of the arc, with odd (respectively even) index $i$,  if the arc crosses an initial (respectively terminal) arc first.
\item The $g$-vector of a closed loop is equal to $\sum (e_{\tau_i}-e_{\zs_i})$, where the sum is taken over all odd maximal $T$-fans of the loop, and $\zs_i$ (respectively $\tau_i$) is the initial (respectively terminal) arc of the $i$-th fan.
\end{enumerate}
\end{corollary}

\begin{proof}
This follows from Theorem \ref{gbijection} and Lemmas \ref{lem g1} and \ref{lem g2}.
\end{proof}

\section{Coefficient systems coming from a full-rank exchange matrix}\label{full-rank}

In this section we will prove Corollary \ref{cor:coefficients}, which extends the 
results of this paper to a cluster algebra from a surface with a coefficient system 
coming from a full-rank exchange matrix.

Let $(S,M)$ be a surface without punctures and at least two marked points, and let 
$T=(\tau_1,\dots,\tau_n)$
be a triangulation of $(S,M)$.  Let $B$ be a full-rank $m \times n$ exchange matrix,
whose top $n \times n$ part $B_T$ comes from the triangulation $T$.  Let $\Afull
= \A(B) \subset \Q(\X_1,\dots,\X_m)$; here $(\X_1,\dots,\X_n)$ is the set of initial cluster variables.  
We will construct two bases $\BB^{\circ}$ and $\BB$ for $\Afull$,
using the corresponding bases $\B^{\circ}$ and $\B$ for $\A$, where $\A$ is the 
cluster algebra associated to $(S,M)$ with principal coefficients with respect to 
the seed $T$.

In order to define $\BB^{\circ}$ and $\BB$, we first recall the 
\emph{separation formulas} from \cite{FZ4}.  We will apply it here
for the case of the cluster algebra of geometric type $\Afull = \A(B)$.
First we need some notation.

If $P(u_1,\dots,u_n)$ is a Laurent polynomial, we define
$\Trop(P)$ by setting $$\Trop(\prod_j u_j^{a_j} + \prod_j u_j^{b_j}) = 
\prod_j u_j^{\min(a_j,b_j)},$$
and extending linearly.  In particular, $\Trop(P)$ is always a Laurent monomial.

Let $\Sigma_{t_0} = (x_1,\dots,x_n; y_1,\dots,y_n; B_T)$ be the initial seed of 
the cluster algebra with principal coefficients $\A$.  
For each $1 \leq j \leq n$, we define 
$$\Y_j = \prod_{i=n+1}^m \X_i^{b_{ij}} \text{ and }
\hat{\Y}_j = 
\prod_{i=1}^m \X_i^{b_{ij}}.$$
Then \cite[Theorem 3.7]{FZ4} and \cite[Corollary 6.3]{FZ4} express the cluster variable 
$x_{\gamma}$  of $\Afull$ as the following equivalent statements. Recall that
$X_{\gamma}^T$ and $F_{\gamma}^T$ denote the quantities defined in Definitions \ref{def:matching} and 
\ref {def closed loop}, see also Remark \ref{remnot}.

\begin{proposition}\label{separation}
\begin{equation*}
x_{\gamma} = 
\frac{X_{\gamma}^{T}(\X_1,\dots,\X_n; {\Y}_1,\dots,{\Y}_n)}
{\Trop(F_{\gamma}^{T}(\Y_1,\dots,\Y_n))}  =
\frac{F_{\gamma}^{T}(\hat{\Y}_1,\dots,\hat{\Y}_n)}
{\Trop(F_{\gamma}^{T}(\Y_1,\dots,\Y_n))} \cdot \X_1^{g_1} \dots \X_n^{g_n}.
\end{equation*}
Here $(g_1,\dots,g_n)$ is the $\gg$-vector of $X_{\gamma}^T$.
\end{proposition}

By analogy, if $\zeta$ is a closed loop in $(S,M)$, we \emph{define} 
the cluster algebra element $x_{\zeta}$ in $\Afull$ as follows.

\begin{definition}\label{separation2}
\begin{equation*}
x_{\zeta} = 
 \frac{X_{\zeta}^{T}(\X_1,\dots,\X_n; \Y_1,\dots,\Y_n)}
{\Trop(F_{\zeta}^{T}(\Y_1,\dots,\Y_n))} = 
\frac{F_{\zeta}^{T}(\hat{\Y}_1,\dots,\hat{\Y}_n)}
{\Trop(F_{\zeta}^{T}(\Y_1,\dots,\Y_n))} \cdot \X_1^{g_1} \dots \X_n^{g_n},
\end{equation*}
where $(g_1,\dots,g_n)$ is the $\gg$-vector of $X_{\zeta}^T$ 
(see Definition \ref{g-def}).
\end{definition}
Note that it is easy to check that the second and third expressions above
are equivalent, following the proof of \cite[Corollary 6.3]{FZ4}.

Now that we have defined  elements of $\Afull$ associated to each 
arc and closed loop, we may define the collections of elements which will comprise
our bases.  
\[\BB^{\circ} = \left\{\prod_{c\in C} x_c \ \vert \ C \in \C^{\circ}(S,M) \right\}
\text{ and }
\BB = \left\{\prod_{c\in C} x_c \ \vert \ C \in \C(S,M) \right\}.\]
As before, $\C^{\circ}(S,M)$ and $\C(S,M)$ denote the 
$\C^{\circ}$-compatible and $\C$-compatible collections of arcs and loops.

\begin{theorem}
$\BB^{\circ}$ is a basis for $\Afull$.  Similarly,
$\BB$ is a basis for $\Afull$.
\end{theorem}

\begin{proof}
First we show that 
$\BB^{\circ}$ and $\BB$ are subsets of $\Afull.$
We define a homomorphism of algebras
$\phi:\A \to \Afull$ 
which sends
each cluster variable $X_{\gamma}^T$ to $X_{\gamma}^T(\X_1,\dots,\X_n; \Y_1,\dots,\Y_n)$.
This is just a specialization of variables, so in particular it is a homomorphism.
Using this notation,
\begin{equation}\label{specialize}
x_{\zeta} = \frac{\phi(X_{\gamma}^T)} 
{\Trop(F_{\zeta}^{T}(\Y_1,\dots,\Y_n))},
\end{equation} where the denominator is a Laurent monomial
in coefficient variables.  Therefore whenever $X_{\zeta}^T$ lies in $\A$ -- 
i.e. whenever $X_{\zeta}^T$  can be written as a polynomial in cluster variables -- then 
$x_{\zeta}$ can also be written as a polynomial in cluster variables and hence is in $\Afull$.
Since we have shown that $\B^{\circ}$ and $\B$ are subsets of $\A$, it follows
that $\BB^{\circ}$ and $\BB$ are subsets of $\Afull.$

Next we show that $\BB^{\circ}$ and $\BB$ are spanning sets for $\Afull$.  As before,
each $k$-bracelet $x_{\Brac_k}(\zeta)$ 
can be written as a Chebyshev polynomial in $x_{\zeta}$, so 
it suffices to show that $\BB^{\circ}$ spans $\Afull$.  By the arguments of the previous
paragraph and \eqref{specialize}, 
every skein relation in $\A$ gives rise to a skein relation in $\Afull$.  It follows that
we can write every polynomial in cluster variables in terms of the elements of 
$\BB^{\circ}$.

Finally we show that the elements of 
$\BB^{\circ}$ (respectively $\BB$) are linearly independent.
Every $F$-polynomial $F_{\gamma}^T$ and $F_{\zeta}^T$ has constant term $1$.
Therefore it follows from Proposition \ref{separation} and Definition \ref{separation2}
that the Laurent expansion of any element
$x_{\gamma}$ (respectively $x_{\zeta}$) contains a Laurent monomial 
$x_1^{g_1} \dots x_n^{g_n} x_{n+1}^{g_{n+1}} \dots x_m^{g_m}$,
where $(g_1,\dots,g_n)$ is the $\gg$-vector of $x_{\gamma}$ (resp. $x_{\zeta}$), 
and the exponent vector of any other Laurent monomial in the same expansion
is obtained from $(g_1,\dots,g_m)$ by adding some nonnegative integer linear
combination of the columns of $B$.  The same property holds for 
monomials in the variables $x_{\gamma}$ and $x_{\zeta}$.  
Therefore by Theorem \ref{th:bijection} (which shows that the $\gg$-vectors 
are all distinct) and Proposition \ref{suffice}, the elements
of $\BB^{\circ}$ are linearly independent.  Similarly
for $\BB$. 
\end{proof}

\settocdepth{section}
\section{Appendix: Extending the results to surfaces with punctures}\label{appendix}

In this section we explain how the results and proofs in this paper need to 
be modified when dealing with a marked surface $(S,M)$ which has punctures,
i.e. marked points in the interior of $S$.  In the presence of punctures,
cluster variables are in bijection with \emph{tagged arcs}, 
which generalize ordinary arcs, and clusters are in bijection
with \emph{tagged triangulations}.  
In this section we will assume that
the reader is familiar with tagged arcs;
see \cite[Section 7]{FST} for details.
If $\gamma$ is an arc (without notches) with an endpoint at puncture $p$,
we denote the corresponding tagged arc which is notched at $p$ by 
$\gamma^{(p)}$.  If $\gamma$ is an arc (without notches) with endpoints
at punctures $p$ and $q$, we denote the corresponding
tagged arc which is notched
at both those punctures by $\gamma^{(pq)}$.

We believe that the results of the present paper may be extended to the 
case of marked surfaces $(S,M)$ which have punctures.  The main obstacle
is to prove the appropriate skein relations for tagged arcs, 
using principal coefficients, and to extend Lemma \ref{skein-g}
 to this setting.  We will give several approaches to doing so 
at the end of Section \ref{sec:spanning}.  We believe that the second
approach described there is most plausible; the drawback is that 
it involves giving separate proofs for all fifteen cases of the new
tagged skein relations.

\subsection{Definition of $\B^{\circ}$ and $\B$}
Our definitions of the conjectural bases 
are just a slight generalization of the corresponding definitions
from Section \ref{sec:bangbrac}.

\begin{definition}
A closed loop in $(S,M)$ is called \emph{essential} if
  it is not contractible \emph{nor contractible onto a single puncture},
and it does not have self-crossings.
\end{definition}

\begin{definition}
A collection $C$ of tagged arcs and essential loops is called 
\emph{$\C^{\circ}$-compatible} if the 
 tagged arcs in $C$ are pairwise compatible, and 
no two elements of $C$ cross each other.
We define $\C^{\circ}(S,M)$  to be the set of all 
$\C^{\circ}$-compatible collections in $(S,M)$.

A collection $C$ of tagged arcs and bracelets is called 
\emph{$\C$-compatible} if:
\begin{itemize}
\item the tagged arcs in $C$ are pairwise compatible;
\item no two elements of $C$ cross each other except for the self-crossings of a bracelet; and
\item given an essential loop $\gamma$ in $(S,M)$, 
there is at most one $k\ge 1$ such
that the $k$-th bracelet $\Brac_k\gamma$ lies in $C$, and, moreover, there is at
most one copy of this bracelet $\Brac_k\gamma$ in $C$.
\end{itemize}
We define $\C(S,M)$ to be the set of all $\C$-compatible
collections in $(S,M)$.
\end{definition}

\begin{definition}
We define $\B^\circ$ 
to be the set of all cluster algebra 
elements in $\A = \Aprin(B_T)$ corresponding to the set $C^{\circ}(S,M)$, that is,
\[\B^{\circ} = \left\{\prod_{c\in C} x_c \ \vert \ C \in \C^{\circ}(S,M) \right\}.\]
Similarly, we define
\[\B = \left\{\prod_{c\in C} x_c \ \vert \ C \in \C(S,M) \right\}.\]
\end{definition}

\subsection{Cluster algebra elements
associated to generalized tagged arcs}\label{gen-tagged-arcs}

In order to prove that $\B^{\circ}$ and 
$\B$ are spanning sets, we need to prove skein relations
involving tagged arcs. As in the unpunctured case, the skein
relation involving  tagged arcs
should have a simple pictorial description in terms of resolving
a crossing.  However, when one resolves two (tagged) arcs that cross each other
more than once, one may get a generalized (tagged) arc, that is, a
(tagged) arc
with a self-crossing.  See Figure \ref{figD4examp}.
For this reason we need to make sense
of the element of the (fraction field of the) cluster algebra associated to a generalized tagged arc.
As in \cite{MSW}, in order to deduce the positivity of such elements
with respect to all clusters, it suffices to consider 
cluster algebras of the form $\Aprin(B_T)$, where $T$ is an 
\emph{ideal triangulation} of $(S,M)$.  (Note that the snake graph or band graph corresponding to an arc can be defined even if it crosses through self-folded triangles.)

\begin{figure}
\scalebox{0.3}
{
\input{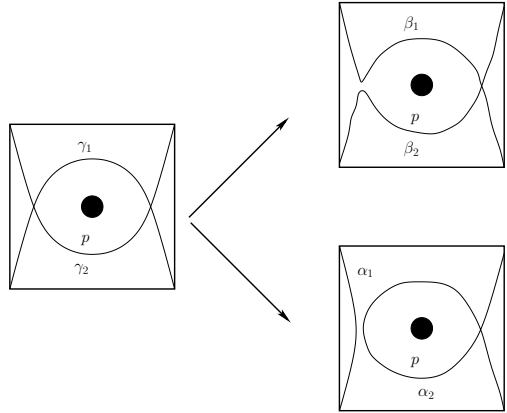}
}
\caption{Smoothing two arcs may produce a generalized arc  
with a self-crossing.}
\label{figD4examp}
\end{figure}

There are several options for how to define the elements
$x_{\gamma^{(p)}}$ and $x_{\gamma^{(pq)}}$, when $\gamma^{(p)}$ 
and $\gamma^{(pq)}$ are
generalized tagged arcs.  All three options should be equivalent.

\begin{enumerate}
\item Algebraic definition.  
If $\gamma$ is an arc (without self-crossings), with 
one end incident to a puncture $p$, then 
$x_{\ell} = x_{\gamma} x_{\gamma^{(p)}}$, where $\ell$ is the arc
cutting out a once-punctured monogon enclosing $p$ and $\gamma$.
If $\gamma$ is an arc (without self-crossings)
between two punctures $p$ and $q$, then 
there is a more complicated identity  (see \cite[Theorem 12.9]{MSW})
that expresses $x_{\gamma^{(pq)}}$ in terms of 
$x_{\gamma}$, $x_{\gamma^{(p)}}$, and $x_{\gamma^{(q)}}$.
By analogy, if $\gamma$ is a generalized arc (with self-crossings
allowed), then one could define $x_{\gamma^{(p)}}$ and 
$x_{\gamma^{(pq)}}$ using the above algebraic identities.

\item Combinatorial definition.
In \cite[Theorem 4.16]{MSW} and \cite[Theorem 4.20]{MSW}, we proved
that the cluster algebra elements associated to singly and doubly-notched
arcs $x_{\gamma^{(p)}}$ and $x_{\gamma^{(pq)}}$ have Laurent expansions
which are given as sums over \emph{$\gamma$-symmetric matchings}
and \emph{$\gamma$-compatible pairs of matchings}, respectively.
By analogy, when $\gamma$ is a generalized arc with self-intersections,
one could define 
$x_{\gamma^{(p)}}$ and $x_{\gamma^{(pq)}}$ combinatorially, 
in terms of 
$\gamma$-symmetric matchings
and $\gamma$-compatible pairs of matchings.
The proofs of \cite[Section 12]{MSW} should carry over and show that 
the above algebraic and combinatorial definitions of 
$x_{\gamma^{(p)}}$ and $x_{\gamma^{(pq)}}$ are equivalent.

\item Definition using the separation formula.
The \emph{separation formula} (\cite[Theorem 3.7]{FZ4})
expresses the cluster variables of a cluster algebra
over an arbitrary semifield, with a seed at $t_0$,
using the cluster variables and F-polynomials 
of the corresponding cluster algebra with principal coefficients
at $t_0$.  By using the separation formula -- together with 
the fact that the $B$-matrix of a tagged triangulation 
equals the $B$-matrix of a corresponding ideal triangulation
(obtained by changing the tagging around a collection 
of punctures) -- one obtains a formula for cluster variables
associated to ordinary arcs, in cluster algebras 
$\Aprin(B_T)$, where $T$ is an arbitrary tagged triangulation.
One may then combine this formula with \cite[Proposition 3.15]{MSW},
in order to obtain a formula for cluster variables associated to 
tagged arcs, in cluster algebras
$\Aprin(B_T)$, where $T$ is an arbitrary ideal triangulation.
By analogy, when $\gamma$ is a generalized arc, one could \emph{define} 
$x_{\gamma^{(p)}}$ and $x_{\gamma^{(pq)}}$ by extending the above 
formula from tagged arcs to generalized tagged arcs.
\end{enumerate}

\subsection{Cluster algebra elements associated to closed loops}

A closed loop is not incident to any marked points, thus there is no such thing as a tagged closed loop.  We therefore define $X_\zeta^T = x_\zeta$ when $\zeta$ is a closed curve via good matchings in a band graph, just as before (Definition \ref{def closed loop}), with one exception.  If $\zeta$ is a closed loop without self-intersections enclosing a single puncture $p$, then 
$X_\zeta^T = 1 + \frac{y_\tau}{y_\tau^{(p)}}$ or $1 + \prod_{\tau \in T} y_\tau^{e_p(\tau)}$, depending on whether $T$ contains a self-folded triangle containing $p$ or not.  Here, $e_p(\tau)$ denotes the number of ends of $\tau$ incident to $p$.

\subsection{$\B^{\circ}$ and $\B$ are spanning sets for $\A$}\label{sec:spanning}
In order to prove that both $\B^{\circ}$ and $\B$ span $\Aprin(S,M)$, one must prove skein relations 
involving tagged arcs. Note that two tagged arcs are incompatible
if they cross each other, or if they have an incompatible tagging at a
puncture, as in the left-hand side of 
Figure \ref{fig tagged skein}.
\begin{figure}
\scalebox{1.0}
{
\input{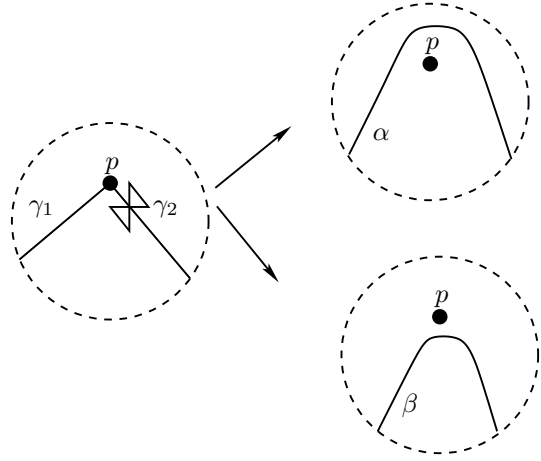}
}
\caption{Resolving an incompatibility at a puncture}\label{fig tagged skein}
\end{figure}

In particular, one must prove skein relations involving:
\begin{enumerate}
\item an ordinary arc and a singly-notched arc, which cross each other
\item an ordinary arc and a doubly-notched arc, which cross each other
\item two singly-notched arcs, which cross each other
\item a singly-notched arc and a doubly-notched arc, which cross each other
\item two doubly-notched arcs, which cross each other
\item an ordinary arc and a singly-notched arc, which have an 
incompatible tagging at a puncture
\item an ordinary arc and a doubly-notched arc, which have one   incompatible tagging at a puncture
\item an ordinary arc and a doubly-notched arc, which have 
two incompatible taggings at a puncture
\item two singly-notched arcs, which have one   incompatible
tagging at a puncture
\item two singly-notched arcs, which have   two incompatible
taggings at a puncture
\item a singly-notched arc and a doubly-notched arc, which have
an incompatible tagging at a puncture
\item a singly-notched arc and a loop
\item a doubly-notched arc and a loop
\item a singly-notched generalized arc with a self-crossing
\item a doubly-notched generalized arc with a self-crossing
\end{enumerate}

In the coefficient-free case, proving skein relations
is straightforward.  One may use the fact that given a puncture
$p$ in $M$, the map $\Psi_p$ which sends an arc $\gamma$ to 
either $\gamma^{(p)}$ or $\gamma$ (depending on whether $\gamma$
has an endpoint at $p$ or not) induces an automorphism 
on the cluster algebra $\A(B_T) = \A(S,M)$. 
 This automorphism maps the cluster
corresponding to the triangulation $T$ to the cluster corresponding to the triangulation $T'$ obtained
from $T$ by changing the tags at the puncture $p$, and it is easy to show that it commutes with the
mutations at these clusters; note that this is a cluster automorphism in the sense of \cite{ASS}.
This reduces all of the tagged skein relations 
involving a crossing, to the untagged skein relations that
we have already proved.  
Additionally, proving the skein relation from 
Figure \ref{fig tagged skein} involving an ordinary
arc and a singly-notched arc with an incompatible tagging at a 
puncture, is straightforward, using the identity 
$x_{\gamma} x_{\gamma^{(p)}} = x_{\ell}$ together with 
an ordinary skein relation (the same proof works with 
principal coefficients, as well). 
Similar proofs should
work for all other skein relations involving an incompatible
tagging at a puncture, at least in the coefficient-free case.
Note that Fock
and Goncharov proved that $\mathcal{B}$ is a basis 
of the upper cluster algebra in the
coefficient-free case, even in the presence of punctures, see
\cite[Section 12.6]{FG1}, by utilizing the monodromy around
punctures.

However, in the presence of principal coefficients, 
the map $\Psi_p$ is not a cluster automorphism on $\Aprin(B_T)$;
it acts nontrivially on the coefficients.  Therefore
it is not possible, as above, to use this map 
to reduce the tagged skein relations involving a crossing,
to the corresponding untagged skein relations.

Additionally, we do not know a good analogue of 
the matrix formulas in \cite{MW} for cluster variables associated to 
arcs with notches.  If one had such matrix formulas, one might hope to
prove the corresponding skein relations via matrix identities,
as in \cite{MW}.

There are several alternative approaches that one might use.  
A first approach is to use the formulas and definitions of 
Section \ref{gen-tagged-arcs} (3) (the separation formula), in order to 
prove the tagged skein relations.  This approach allows us 
to express the cluster algebra elements associated to 
tagged arcs and tagged generalized arcs, in terms of the cluster
variables and F-polynomials associated to untagged
arcs and generalized arcs.  From such formulas, one could obtain
some ``skein relations" immediately.  However, using this approach,
it is not at all clear how to prove the analogue of Lemma \ref{skein-g}.

A second approach is to use the algebraic identities that 
the cluster algebra elements associated to tagged arcs satisfy.
For example, if one wants to prove the skein relation involving
an ordinary arc $x_{\gamma_1}$ and a singly-notched arc $x_{\gamma_2^{(p)}}$
which cross each other, one could use the identity 
$x_{\gamma_2} x_{\gamma_2^{(p)}} = x_{\ell_0}$.   
By considering the skein relation involving $x_{\gamma_1}$ and 
$x_{\ell_0}$, and keeping careful track of the coefficients using 
the lamination corresponding to the initial triangulation $T$,
it is possible to write down the skein relation that expresses
$x_{\gamma_1} x_{\gamma_2^{(p)}}$.  

\begin{example}
\label{example-tagged}

\begin{figure}
\input{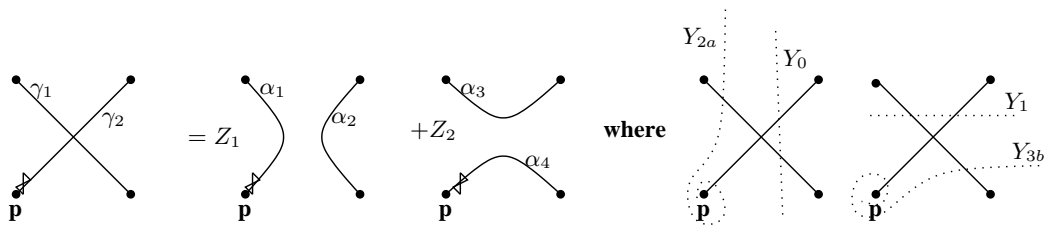}
\caption{Illustrating Example \ref{example-tagged}}
\label{tag1}
\end{figure}

(Case (1) of the skein relations) Let $\alpha_1, \alpha_2, \alpha_3,$ and $\alpha_4$ be the four arcs
obtained by smoothing at the intersection point of $\gamma_1$ and $\gamma_2$, 
as in Figure \ref{tag1}.  Then there are monomials in the coefficient variables, $Z_1$ and $Z_2$, such that
\begin{equation}
\label{tagged-skein}
x_{\gamma_1} x_{\gamma_2^{(p)}} = Z_1 x_{\alpha_1^{(p)}} x_{\alpha_2}+
Z_2 x_{\alpha_3} x_{\alpha_4^{(p)}},
\end{equation}
and precisely one of them equals $1$.
\end{example}

\begin{proof}To show this, we will show that 
$Z_1 = Y_{0} Y_{2a}$ and $Z_2 = Y_1 Y_{3b}$, 
where $Y_0$, $Y_1$, $Y_{2a}$ and $Y_{3b}$ 
are monomials in coefficient variables representing contributions
from the laminations whose local configurations are as shown
by the dotted curves in Figure \ref{tag1}. Note that 
we use the subscript ``a'' (resp. ``b") to indicate a contribution from 
laminations spiraling counterclockwise (resp. clockwise) into the puncture. 

We multiply both sides of \eqref{tagged-skein} by $x_{\gamma_2}$ and verify the resulting equation.
Applying skein relations to $x_{\gamma_2}$ times 
the left-hand-side of \eqref{tagged-skein}, that is, to  
$x_{\gamma_2} x_{\gamma_1} x_{\gamma_2^{(p)}} =x_{\gamma_1} x_{\ell_0}$, we get
\begin{align}
x_{\gamma_1} x_{\ell_0} &=Y_1 x_{\alpha_3} x_{\beta_0} + Y_{2a} Y_{2b} Y_0 x_{\beta_1} x_{\alpha_2}\\ 
\label{eq-lhs}
&= Y_1 Y_{3a} Y_{3b} x_{\alpha_3} x_{\beta_2} +
Y_0 Y_1 Y_4 x_{\alpha_2} x_{\alpha_3} x_{\omega} +
Y_0 Y_{2a} Y_{2b} x_{\alpha_2} x_{\beta_1},
\end{align}
where the (generalized) arcs $\beta_0, \beta_1$ and $\beta_2$
and the closed loop $\omega$
are as in Figure \ref{tag2}.  Also, 
$Y_{2a}, Y_{2b}, Y_{3a}, Y_{3b},$ and $ Y_4$ are monomials in coefficient
variables representing contributions from laminations whose
local configurations are as shown by the dotted curves in Figure \ref{tag2}.

\begin{figure}
\input{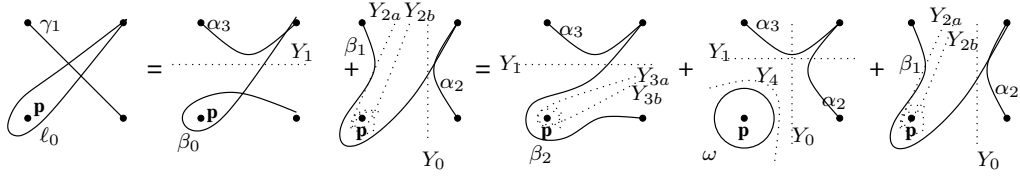}
\caption{Left-hand side of Equation (\ref{tagged-skein})}\label{tag2}
\end{figure}

On the right hand side of 
\eqref{tagged-skein}, after multiplying through by $x_{\gamma_2}$, we obtain 
\begin{align*}
x_{\gamma_2} x_{\alpha_1^{(p)}} x_{\alpha_2} &=
      x_{\ell_1} x_{\beta} x_{\alpha_2} (x_{\alpha_1})^{-1} \\
   &= (Y_{2b} x_{\alpha_1} x_{\alpha_2} x_{\beta_1} + 
    Y_1 Y_{3a} Y_4 Y_{5a} x_{\alpha_1} x_{\alpha_2} x_{\alpha_3})(x_{\alpha_1})^{-1} \\
   &= Y_{2b} x_{\alpha_2} x_{\beta_1} + Y_1 Y_4 Y_{3a} Y_{5a} x_{\alpha_2} x_{\alpha_3},
\end{align*}
see Figure \ref{tag3}. Here $Y_{5a}$ represents the contribution 
from all leaves spiraling counterclockwise into $p$ which are not already
included in $Y_{2a}$ and $Y_{3a}$.
\begin{figure}
\input{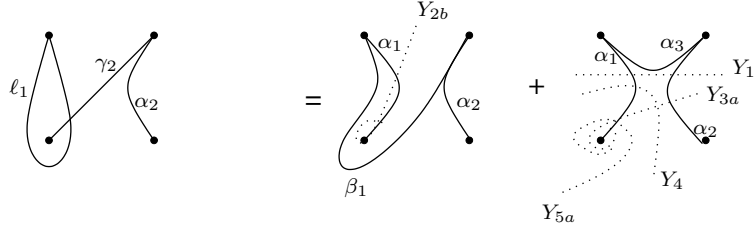}
\caption{First term on right-hand side of Equation (\ref{tagged-skein})}\label{tag3}
\end{figure}

Similarly, 
using the notation of Figure \ref{tag4}, we get
\begin{equation*}
x_{\gamma_2} x_{\alpha_3} x_{\alpha_4^{(p)}} =
   x_{\gamma_2} x_{\alpha_3} x_{\ell_2} (x_{\alpha_4})^{-1}=
 Y_{3a} x_{\alpha_3} x_{\beta_2} + Y_0 Y_4 Y_{2b} Y_{5b} x_{\alpha_3}
  x_{\alpha_2}.
\end{equation*}
\begin{figure}
\input{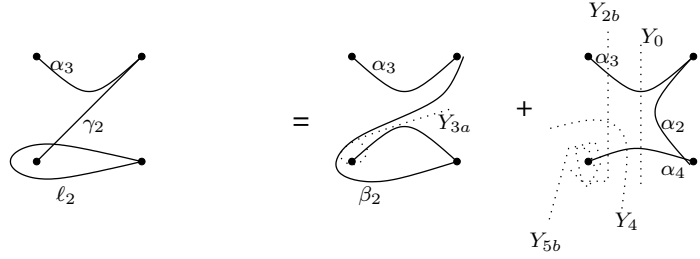}
\caption{Second term on right-hand side of Equation (\ref{tagged-skein})}\label{tag4}
\end{figure}
Therefore $x_{\gamma_2}$ times the right-hand-side of \eqref{tagged-skein}
is equal to 
\begin{equation}
\label{eq-rhs}
Z_1 Y_{2b} x_{\beta_1} x_{\alpha_2}+
Z_2 Y_{3a} x_{\alpha_3} x_{\beta_2} +
(Z_1 Y_1 Y_4 Y_{3a} Y_{5a} + Z_2 Y_0 Y_4 Y_{2b} Y_{5b}) x_{\alpha_2}
x_{\alpha_3}.
\end{equation}
We need to show that (\ref{eq-rhs}) = (\ref{eq-lhs}).

Setting $Z_1 = Y_0 Y_{2a}$ and $Z_2 = Y_1 Y_{3b}$
makes two terms in each of the above expressions coincide,
and we've reduced the proof of \eqref{tagged-skein} to showing that
$Y_0 Y_{2a} Y_1 Y_4 Y_{3a} Y_{5a} + Y_1 Y_{3b} Y_0 Y_4 Y_{2b} Y_{5b}=
Y_0 Y_1 Y_4 x_{\omega},$ or equivalently,
\begin{equation} \label{reduction}
Y_{2a} Y_{3a} Y_{5a} + Y_{3b} Y_{2b} Y_{5b} = x_{\omega}.
\end{equation}

There are two cases, based on whether
$T$ contains a self-folded triangle enclosing the puncture $p$.
If not, then all leaves of the lamination spiral counterclockwise
into $p$, and so $Y_{2b} Y_{3b} Y_{5b} = 1$.
In this case, it follows from the definition that 
$x_{\omega} = 1+Y_{2b} Y_{3b} Y_{5b}$
(since the second monomial represents the product of all coefficient
variables indexed by arcs of $T$ incident to $p$).
This proves \eqref{reduction}.

If $T$ does contain a self-folded triangle enclosing
puncture $p$, then let us denote the radius incident to $p$ by $r$.
In this case there are exactly two leaves of the lamination
spiraling into $p$, $L_r$ and $L_{r^{p}}$, which spiral counterclockwise
and clockwise, respectively.
In this case the left-hand-side of \eqref{reduction}
equals $y_r + y_{r^{(p)}}$.  But this agrees with the 
definition of  $x_{\omega}$.  Either way, we have now shown \eqref{tagged-skein}.

Now, we claim that at least one of $Y_0$ and $Y_1$ is not equal to $1$.
If both were $1$, then any laminations cutting across the quadrilateral
formed by the endpoints of $\gamma_1$ and $\gamma_2$ would have to 
cut across corners of the quadrilateral.  But such a lamination
could not have come from a triangle.
Now note that if $Y_1 \neq 1$ then $Y_0$ and $Y_{2a}$ must equal $1$,
since the leaves of a lamination cannot intersect each other.
Similarly, if $Y_0 \neq 1$, then $Y_1$ and $Y_{3b}$ must equal $1$.\end{proof}

We have shown how to prove the first of fifteen skein relations,
and prove the analogue of Lemma   \ref{skein-g} for this case.  In theory, one may
give a similar argument on a case-by-case basis for the remaining fourteen types of skein relations above. 
 \emph{We believe that this approach would successfully
generalize the results of the present paper to the 
case of general surfaces $(S,M)$, with or without punctures.}

\subsection{$\B^{\circ}$ and $\B$ are linearly independent sets}

If one can extend Lemma \ref{skein-g}
to the case of 
tagged arcs, then it is possible to prove  that 
the sets $\B^{\circ}$ and 
$\B$ are linearly independent.

Indeed, one may extend Proposition \ref{invert}
to define a tagged arc $\overline{\tau}_i$ of $(S,M)$
such that $\gg(x_{\overline{\tau}_i})=-e_i$ for each 
$1 \leq i \leq n$.  One may call this the \emph{anti-arc} construction.

\begin{itemize}
\item If $\tau_i$ is an arc between two marked points,
both on a boundary component, then the definition of 
$\overline{\tau}_i$ is the same as in Proposition \ref{invert}.
\item Suppose that $\tau_i$ is an arc between two marked points $x$ and $p$,
where $x$ lies on a boundary component and $p$ is a puncture.
Let $d_1$  denote the boundary
segment such that  $d_1$ is incident to $x$ and is in the
clockwise direction from $\tau_i$; let $x'$ denote the other endpoint
of $d_1$.  Let  
$\overline{\tau}_i$ 
be the tagged arc of $(S,M)$ between points
$x'$ and $p$, which is tagged plain at $x'$ and notched at $p$,
such that its untagged version is homotopic to the concatenation of
$d_1$ and $\tau_i$.
\item Suppose that $\tau_i$ is an arc between two punctures $p$ and $q$.
Let $\overline{\tau}_i$ 
be the tagged arc of $(S,M)$ which is obtained from
$\tau_i$ by notching both ends.
\end{itemize}
In order to 
prove that $\gg(x_{\overline{\tau}_i})=-e_i$,
one uses the tagged skein relations.

It is then straightforward to extend the arguments of 
Section \ref{sec:g-inverse}, to show that in 
almost all cases, the $\gg$-vector maps 
$\gg: \B^{\circ} \to \Z^n$ and 
$\gg: \B \to \Z^n$ are bijections.
A main tool here is the generalization of 
Lemma \ref{skein-g}.
The only situation in which the 
$\gg$-vector map is not a bijection to $\Z^n$ is the case
that $(S,M)$ is a once-punctured closed surface.  In this case
$\gg$ is an injection but not a surjection. (This is because the 
anti-arc construction for such a surface always gives a doubly-tagged arc, which is 
in the tagged arc complex but not the cluster complex, when 
$(S,M)$ is a once-punctured closed surface.)  However, 
injectivity suffices to show linear independence: by the proof of 
Corollary \ref{cor:independence}, 
it is enough to know that the $\gg$-vectors 
of the basis elements are all distinct.

\subsection{$\B^{\circ}$ and $\B$ are subsets of $\A$}

One can show that the bases $\B^{\circ}$ and $\B$ are 
subsets of $\Aprin$, if $S$ has a non-empty boundary and 
at least two of its marked points are on the boundary, or if 
$S$ has genus zero.  It suffices to show that the cluster algebra
elements corresponding to essential loops lie in $\Aprin$.

The proof of Proposition \ref{lem:containment} (which treats the case
when at least two marked points are on the boundary) goes through 
without changes in the presence of punctures. 

However, when $(S,M)$ has punctures, a new 
argument is required in order to prove Corollary \ref{cor:containment} (which
treats the case that $S$ has genus zero).  
Let $\zeta$ be an essential loop
that cuts out a disk with at least two punctures
$m_1$ and $m_2$
inside it. If $S$ is a sphere, then $\zeta$ cuts out two disks, and
we choose the one with the smaller number of punctures inside it.
One can then prove Corollary \ref{cor:containment} by induction on the 
number of punctures inside $\zeta$.  
The idea is to consider an appropriate skein 
relation involving an unnotched arc between $m_1$ and $m_2$,
and a doubly notched arc between $m_1$ and $m_2$.

{}

\end{document}